\pdfoutput=1

\documentclass{amsart}

\usepackage{todonotes}
\usepackage{enumitem}
\usepackage{verbatim}
\usepackage{xcolor}
\usepackage{tikz}
\usepackage{tikz-cd}
\usepackage{arydshln}
\usepackage{caption}
\usepackage{subcaption}
\usepackage{graphicx}
\usepackage[hyphens]{url}
\usepackage{amsmath}
\usepackage{bbm}
\usepackage{hyperref}
\usepackage{mathrsfs}
\usepackage{todonotes}
\usepackage{amssymb}
\usepackage{mathabx}

\newtheorem{theorem}{Theorem}[section]
\newtheorem{proposition}[theorem]{Proposition}
\newtheorem{lemma}[theorem]{Lemma}
\newtheorem{sublemma}[theorem]{Sublemma}
\newtheorem{corollary}[theorem]{Corollary}

\newtheorem{warning}[theorem]{Warning}

\theoremstyle{definition}
\newtheorem{definition}[theorem]{Definition}

\theoremstyle{remark}
\newtheorem{remark}[theorem]{Remark}
\newtheorem{question}[theorem]{Question}

\numberwithin{equation}{section}

\newcommand{\cP}{\mathcal{P}}
\newcommand{\cO}{\mathcal{O}}
\newcommand{\cR}{\mathcal{R}}

\newcommand{\bZ}{\mathbb{Z}}
\newcommand{\cC}{\mathcal{C}}

\newcommand{\diam}{\mathrm{diam}}

\renewcommand{\bold}[1]{\medskip \noindent {\bf #1 }\nopagebreak}

\begin{document}

\title[Spheres in the curve graph]{Spheres in the curve graph and linear connectivity of the Gromov boundary}

\author{Alex Wright}
\email{alexmw@umich.edu}

\begin{abstract}
We consider the curve graph in the cases where it is  not a Farey graph, and show that its Gromov boundary is linearly connected.  
For a fixed center point $c$ and radius $r$, we define the sphere of radius $r$ to be the induced subgraph on the set of vertices of distance $r$ from $c$. We show that these spheres are always connected in high enough complexity, and prove a slightly weaker result for low complexity surfaces. 
\end{abstract}

\maketitle

\thispagestyle{empty}

\setcounter{tocdepth}{1}
\tableofcontents



\section{Introduction}

\subsection{A glimpse of the broad context.} One of the central objects in low-dimensional topology and geometric group theory is the complex of curves. This finite dimensional simplicial complex was introduced by Harvey as an analogue of the Tits building for symmetric spaces \cite{Harvey}. From its introduction as a tool to understand the boundary structure of Teichm\"uller space, the scope and utility of its study  broadened in  topology \cite{Harer, Harer2} and  rigidity \cite{Ivanov}. Later, work of Masur and Minsky established that the curve graph itself had tractable geometry, being a Gromov hyperbolic space \cite{MMI}, and work of many authors established beyond any doubt the curve complex as a key tool for understanding the geometry of mapping class groups, Teichm\"uller spaces, and hyperbolic three manifolds  \cite{MMII,ICM1, CombModel, Rank, BKMM, ELC, QT}. Skipping ahead to the present, the uses of the curve complex are now too numerous to recall here. Via the notion of a hierarchically hyperbolic space \cite{HHS1, HHS2, WhatIs} as well as other generalizations and analogies \cite{Tight,Osin, WhatIs2, FSn, FFn, ICM2, KK} the influence of the curve complex has extended  to many important spaces and groups far beyond what might have been expected.

For most purposes outside of algebraic topology it suffices to consider the 1-skeleton of the curve complex, which is known as the curve graph. Motivated by questions on mapping class groups discussed in more detail below, we study here connections between the fine and coarse geometry of the curve graph, discovering new results and recovering and strengthening important previous results that had not been previously understood from the perspective of the curve graph. 

\subsection{Main results.}Let $\Sigma=\Sigma_{g,n}$ be a connected surface with genus $g$ and $n$ punctures. We always assume $g$ and $n$ are such that $\Sigma$ has a complete, finite volume hyperbolic metric and that $(g,n)\neq (0,3)$. We define the complexity of $\Sigma$ as $\xi(\Sigma)=3g-3+n$. We say $\Sigma$ is 
\begin{itemize}
\setlength\itemsep{0.15em}
\item \emph{exceptional} if $\xi(\Sigma)=1$, i.e. $(g,n)\in \{(1,1),(0,4)\}$,
\item \emph{low complexity} if $\xi(\Sigma)=2$, i.e. $(g,n)\in \{(1,2),(0,5)\}$, 
\item \emph{medium complexity} if $\xi(\Sigma)=3$, i.e. $(g,n)\in \{(2,0), (1,3), (0,6)\}$, 
\item \emph{high complexity} if $\xi(\Sigma)\geq 4$. 
\end{itemize}

We let $\cC \Sigma$ be the curve graph of $\Sigma$,  and note the exceptional isomorphisms $\cC \Sigma_{(1,1)} \cong \cC \Sigma_{(0,4)}$,  $\cC \Sigma_{(1,2)} \cong \cC \Sigma_{(0,5)}$, and $\cC \Sigma_{(2,0)} \cong \cC \Sigma_{(0,6)}$. 

Fix an arbitrary vertex $c$ in $\cC \Sigma$, which we sometimes refer to as the center point. For $r\geq 0$ integral, we let $S_r=S_r(c)$ be the sphere of radius $r$ in the graph $\cC\Sigma$, which consists of all the vertices at distance $r$ from $c$. We say that a subset of vertices of a graph is connected if the induced subgraph is connected. 

\begin{theorem}\label{T:SphereConnected}
For all $r$, we have the following: 
\begin{itemize}
\setlength\itemsep{0.15em}
\item If $\Sigma$ is high complexity, $S_r$ is connected. 
\item If $\Sigma$ is medium complexity,   $S_r(c)\cup S_{r+1}(c)$ is connected. 
\item If  $\Sigma$ is low complexity,  $S_r(c)\cup S_{r+1}(c)\cup S_{r+2}(c)$  is connected. 
\end{itemize}
\end{theorem}


We will also address the ``sphere at infinity" as follows. 

\begin{theorem}\label{T:LinearlyConnected}
If $\Sigma$ is non-exceptional, the Gromov boundary of $\cC \Sigma$ is linearly connected. 
\end{theorem}

Here it is implicit that we are using a visual metric on the Gromov boundary. Recall that a metric space $(X,d)$ is said to be linearly connected if there is a constant $L>0$ such that for each pair $x,y\in X$ there is a compact connected set in $X$ containing $x$ and $y$ of diameter at most $L d(x,y)$. Linear connectivity is also called bounded turning or $LLC(1)$, and appears in questions relating to the existence of quasi-isometrically embedded hyperbolic planes in hyperbolic spaces \cite{Tukia, BK, Mackay, MackaySisto}. In our proof, the compact connected set produced will be a path from $x$ to $y$, and as discussed in \cite[Section 1]{Mackay}, this set can in general be taken to be an arc (embedded path) from $x$ to $y$. 

\subsection{Previous results.}
In particular, we obtain a fundamentally new proof of the following result. 

\begin{corollary}\label{C:Gabai}
If $\Sigma$ is non-exceptional, the Gromov boundary of $\cC \Sigma$ is path connected and locally path connected. 
\end{corollary}

Corollary \ref{C:Gabai} was proven in general by Gabai \cite{Gabai1}. Connectivity was proven previously by  Leininger and Schleimer for $\cC \Sigma_{g,n}$ when $g\geq 4$ or when $g\geq 2$ and $n\geq 1$ \cite{LSconnectivity}, and path connectivity and local path connectivity were proven previously by Leininger, Mj, and Schleimer for $\cC \Sigma_{g,1}$ when $g\geq 2$ \cite{LMS}. These results answer questions of Peter Storm recorded in \cite[Question 10]{Shadows} and \cite[Section 2]{MinskyICM}. 

As part of their proof of quasi-isometric rigidity of curve complexes,  Rafi and Schleimer used Corollary \ref{C:Gabai} to deduce that the union of a number of consecutive spheres depending on the hyperbolicity constant is connected \cite[Proposition 4.1]{RS}. Schleimer also proved without using Corollary \ref{C:Gabai} that when $g\geq 2$, complements of balls in $\cC \Sigma_{g,1}$ are connected \cite{SchleimerEnd}.

Following his work on path connectivity, Gabai later established higher connectivity results and, extending work of Hensel and Przytycki for $(g,n)=(0,5)$, proved that boundaries of curve complexes when $g=0$ are homeomorphic to Nobeling spaces \cite{HP, Gabai2}. In light of \cite{Harer, FiniteRigid, BBM}, it is possible that the techniques in Section \ref{S:05} could  provide a starting point for an attempt to reprove and strengthen such higher connectivity results.  We conjecture that Gabai's higher connectivity results could be strengthened to linear higher connectivity, which may be of interest in light work such as \cite{BK2spheres} and may also be compared to work such as \cite{BD} and \cite{BBlocal}. 

The boundary of the curve complex is naturally homeomorphic to the space of ending laminations \cite{klarreich2018boundary}. Much previous work has taken place in the space of ending laminations, including the recent work of Chaika and Hensel proving connectivity results for spaces of uniquely ergodic and cobounded laminations \cite{chaika2019path}.

 \subsection{Structure of the proof.}
 In contrast to much of the work above, we work directly in the curve complex, using the Bounded Geodesic Image Theorem of Masur-Minsky \cite{MMII}, in the spirit of previous work such as the analysis of dead ends by Birman and Menasco \cite{DeadEnds} and the work of Schleimer cited above \cite{SchleimerEnd}. The core approach of this paper might be described briefly as ``push paths away from the center point $c$". To accomplish this, sometimes we use (a weaker statement than) Theorem \ref{T:SphereConnected} for smaller complexity surfaces, so we induct on the complexity of the surface. The base case for the induction on complexity is the five times punctured sphere, where our argument makes use of pentagons and reveals that some points of $S_r$ might almost be thought of as ``closer" to $S_{r+1}$ than others. 

In medium and high complexity our approach requires that we work with ``essentially non-separating curves", which are curves that are either non-separating or go around a pair of punctures. We prove a number of basic results on what we call the essentially non-separating curve graph, which may be useful for other purposes. 

\subsection{Motivation.}
Masur and Minsky famously proved that curve complexes are hyperbolic, and there are now many proofs of this fact \cite{MMI, U1, U2, U3, HPW}. Moreover, the curve complex exhibits very strong hyperbolicity features even beyond its Gromov hyperbolicity. For example, Dowdall, Duchin, and Masur proved that  the ``generic" pair of points on $S_r$ are distance \emph{exactly} $2r$ apart \cite{Spheres}, and wrote that ``In this sense the curve graph is ``even more hyperbolic than a tree."" (This comment seems tailored to comparison to finite valence trees. One might also suggest that in this sense the curve graph is comparably hyperbolic to an infinite valence regular tree.)  

Furthermore, most of the  spaces closely analogous to curve complexes that arise in the modern study of hierarchical hyperbolicity are quasi-trees \cite{HHS1, HHS2, WhatIs}. Thus there seems to be great tension between tree-like behaviour and the connectivity theorems above. This tension is relevant for the elusive and much studied question of whether convex cocompact surface subgroups of mapping class groups exist (see for example \cite{FarbMosher, Schleininger, Mosher2, reid}). 
Linear connectivity in particular is related to this question since convex co-compact surface subgroups gives rise to quasi-isometrically embedded hyperbolic planes in the curve complex \cite{Shadows,hamenstadt2005word}, and proving linear connectivity is, roughly speaking, part way to establishing the existence of many such planes \cite{Tukia, Mackay}. 

The study of the topology of the Gromov boundary of the curve complex is far from complete; for example its topological dimension is not known in general \cite{Gabai2}.  Nonetheless we feel it is worthwhile to draw attention to the rich additional structure on the boundary. Every nice enough hyperbolic space (including curve complexes) can be recovered up to quasi-isometry via a  cone construction from its boundary with a visual metric \cite[Theorem 8.2]{BS}. In contrast, the mere homeomorphism type of the boundary contains vastly less information about the space. See for example \cite{Bourdon} and \cite{MConfDim} for infinite families of hyperbolic groups whose boundaries are the Menger curve but are  pairwise non-quasi-isometric to each other. In some situations, one might even go so far as to say that, far from being the end goal, the homeomorphism type of the boundary is the vessel that holds more refined information on the hyperbolic group. We feel that the situation for the curve complex is even more pronounced, since in addition to metric structure the boundary supports a rich collection of subsurface projection maps. Linear connectivity is, to our knowledge, one of the first results on the metric structure of the boundary, with another notable result being the work of Bestvina and Bromberg on its capacity dimension \cite{BB}.  


\subsection{Open questions.}
Part of the motivation above concerns the following two questions, which we wish to state explicitly now. 
\begin{question}
Do all non-exceptional curve complexes contain quasi-isometrically embedded hyperbolic planes? 
\end{question}
The work of Leininger and Schleimer gives a positive answer in the case of $\Sigma_{g,1}, g\geq 2$ and surfaces that can be obtained via an appropriate branched cover of such surfaces \cite{Schleininger}. 

\begin{question}
In non-exceptional cases, is the cobounded locus in the boundary of of $\cC \Sigma$ quasi-arc connected? 
\end{question}

Motivated by the existence of hierarchy paths and our desire to connect the study of the boundary to subsurface projections, we also propose the following. 

\begin{question}
For each $\Sigma$ that is non-exceptional, does there exist a $D>0$ such that every pair of points in the Gromov boundary can be joined by a path in the Gromov boundary whose projection to the curve complex of each proper subsurface is an un-parametrized $D$-quasigeodesic? 
\end{question}

We believe a positive answer to these questions would signal a significantly improved understanding of the geometry of the curve complex.

The curve graphs of the exceptional surfaces are both the Farey graph, which, being a quasi-tree, has hopelessly disconnected spheres. All of the individual spheres are disconnected in the low complexity case (Corollary \ref{C:disconnected}), but our analysis leaves open the following. 

\begin{question}\label{SSS}
In the medium complexity case, are individual spheres always connected? In the low complexity case, are unions of two consecutive spheres always connected?
\end{question}

Finally, we note that in the study of $\mathrm{Out}(F_n)$, one is interested in the boundary of the free factor complex $FF_n$, which is known to be path connected and locally path connected when $n\geq 18$ by \cite{bestvina2021connectivity}. Since versions of our main tool, the Bounded Geodesic Image Theorem, are available in that context \cite{Submanifold, Subfactor, Taylor}, it would be interesting to see if our methods could be useful for questions such as the following. 

\begin{question}
When is the Gromov boundary of the free factor complex linearly connected? 
\end{question}

 After this paper was initially released, it was discovered than an extension of our analysis can be used to give a positive answer to Question \ref{SSS} \cite{REU}, and our criterion for linear connectivity was improved and applied to fine curve graphs \cite{YL}.

\bold{Acknowledgements.} The author thanks 
 David Gabai, Jonah Gaster, Sebastian Hensel, Jeremy Kahn, John Mackay, Howard Masur, Yair Minsky, Piotr Przytycki, Kasra Rafi,  Saul Schleimer, Alessandro Sisto, and Richard Webb for helpful conversations. The author thanks 
Sayantan Khan for comments on a previous draft. The  author was partially supported by  NSF Grants DMS 1856155 and 2142712 and a Sloan Research Fellowship.

\section{Sufficient conditions for connectivity}\label{S:sufficient}

\subsection{Unions of spheres.}
Theorem \ref{T:SphereConnected} will be proven with the following  sufficient condition for when unions of some number $w$ of consecutive spheres are connected. As in the case of the curve graph, here we use the notation $S_r=S_r(c)$ for the sphere of radius $r$ in a fixed arbitrary graph $\Gamma$, and we also use $B_\rho(z)$ to denote the set of vertices of distance at most $\rho$ from $z$. A path means a sequence of vertices each adjacent to the next.

\begin{lemma}\label{L:BasicConnected}
Let $\Gamma$ be an arbitrary graph, and let $c\in \Gamma$ be arbitrary. Fix $w>0$. Suppose that for any $r$ the following hold.
\begin{enumerate}
\setlength\itemsep{0.15em}
\item\label{BC:V} For every  $z\in S_r(c)$ and  $x, y\in S_{r+1}(c)\cap B_1(z)$ there exists a path $$x=x_0, x_1, \ldots, x_\ell=y$$ with $$x_i\in S_{r+1}\cup \cdots \cup S_{r+w}$$ for $0\leq i\leq \ell$.
\item\label{BC:adjacent} For every adjacent pair $x,y\in S_{r}$ there exists a path 
$$x=x_0, x_1, \ldots, x_\ell=y$$ with $$x_i\in S_{r+1}\cup \cdots \cup S_{r+w}$$ for $0< i< \ell$.
\end{enumerate}
Then for any $r$, the union $S_r \cup \cdots \cup S_{r+w-1}$  is connected. 
\end{lemma}

No bound is assumed for the length $\ell$ of the paths. Later, in Proposition \ref{P:LinearlyConnectedGraph}, we will give a version of this lemma that implies linear connectivity of the Gromov boundary of a hyperbolic graph, which will require additionally that the paths above stay close to their endpoints.

For completeness, we provide a proof of the lemma. 

\begin{proof}
We prove that $S_r \cup \cdots \cup S_{r+w-1}$ is connected by induction on $r$. The base case of $r=0$ is trivial, since every vertex can be connected to $S_0=\{c\}$ by a geodesic. 

For the inductive step, suppose  that  $S_{r}\cup \cdots \cup S_{r+w-1}$ is known to be connected. Hence, any pair $a, b\in S_{r+1}\cup \cdots \cup S_{r+w}$ can be joined by a path in  $S_r\cup \cdots \cup S_{r+w}$. 

For any path $P$ in $S_r\cup \cdots \cup S_{r+w}$, let $V_r(P)$ denote the number of vertices on the path in $S_r$, and let $E_r(P)$ denote the number of edges of the path that go between two vertices of $S_r$. For any $a, b\in S_{r+1}\cup \cdots \cup S_{r+w}$, let $f(a,b)$ denote the minimal value of $V_r(P)+E_r(P)$ over all paths $P$ from $a$ to $b$ in $S_r\cup \cdots \cup S_{r+w}$. So $f(a,b)$ measures the failure for $a$ and $b$ to lie in the same connected component of $S_{r+1}\cup \cdots \cup S_{r+w}$, and our goal is to show that always $f(a,b)=0$. 

Suppose in order to find a contradiction that $S_{r+1}\cup \cdots \cup S_{r+w}$ is not connected, and pick $a, b\in S_{r+1}\cup \cdots \cup S_{r+w}$ such that $f(a,b)$ is positive but as small as possible. Consider a path $P$ in $S_{r}\cup \cdots \cup S_{r+w}$ with vertices 
$$a=q_0, q_1, \ldots, q_m=b$$
for which $V_r(P)+E_r(P)=f(a,b)$.  

First suppose that this path contains a pair $q_j, q_{j+1}$ of consecutive vertices in $S_r$. Applying assumption \eqref{BC:adjacent} with $x=q_j, y=q_{j+1}$, we obtain a path $$q_j=x_0, x_1, \ldots, x_\ell=q_{j+1}.$$
The path $P'$ with vertices
$$q_0, q_1, \ldots, q_{j}, x_1, \ldots, x_{\ell-1}, q_{j+1}, \ldots, q_m$$
has $V_r(P')=V_r(P)$ and $E_r(P')=E_r(P)-1$, giving a contradiction. 

Next suppose that the path does not contain a consecutive pair of vertices in $S_r$. So, if $q_j$ denotes a point on the path in $S_r$, we know that $q_{j-1}, q_{j+1}\in S_{r+1}$. Applying assumption \eqref{BC:V} with $x=q_{j-1}, y=q_{j+1}, z=q_j$, we obtain a path $$q_{j-1}=x_0, x_1, \ldots, x_\ell=q_{j+1}.$$
The path $P'$ with vertices
$$q_0, q_1, \ldots, q_{j-1}, x_1, \ldots, x_{\ell-1}, q_{j+1}, \ldots, q_m$$
has $V_r(P')=V_r(P)-1$ and $E_r(P')=E_r(P)$, giving a contradiction. 
\end{proof}

\subsection{The Gromov boundary}
We now give a sufficient condition for the linear connectivity of the Gromov boundary of a graph. 

\begin{proposition}\label{P:LinearlyConnectedGraph}
Let $\Gamma$ be a Gromov hyperbolic graph. Suppose that $c\in \Gamma$, and that the following hold. 
\begin{enumerate}
\setlength\itemsep{0.15em}
\item\label{P:LC:represented} Every point in the Gromov boundary of $\Gamma$ can be represented by a geodesic ray starting at $c$. 
\item\label{P:LC:1farther} Every vertex of $\Gamma$ is adjacent to  point of $\Gamma$ that is 1 farther from $c$. 
\end{enumerate}
Additionally suppose that there is some $D>0$ such that the following hold for all $r\geq 0$. 
\begin{enumerate}
\setcounter{enumi}{2}
\item\label{PLC:V} For every  $z\in S_r(c)$ and  $x, y\in S_{r+1}(c)\cap B_1(z)$ there exists a path $$x=x_0, x_1, \ldots, x_\ell=y$$ with $$x_i\in (\Gamma-B_r(c))\cap B_D(x)$$ for $0\leq i\leq \ell$.
\item\label{PLC:adjacent} For every adjacent pair $x,y\in S_{r}$ there exists a path $$x=x_0, x_1, \ldots, x_\ell=y$$ with $$x_i\in (\Gamma-B_r(c))\cap B_D(x)$$ for $0<i<\ell$.
\end{enumerate}
Then the Gromov boundary of $\Gamma$ is linearly connected.
\end{proposition}

 This can be compared to \cite[Proposition 3.2]{BestvinaMess} in the context of hyperbolic groups.   A significant part of the proof is inspired by \cite[Proof of Proposition 5.2]{MackaySisto}.

\begin{proof}
Let $\cR$ be the set of half-infinite geodesic rays  beginning at $c$. A point of $\cR$ can be viewed as a function $X:\{0,1,2,\ldots\}\to \Gamma$ with $X(i)\in S_i(c)$ and $X(i)$ adjacent to $X(i+1)$. 

The first assumption gives that the Gromov boundary is a quotient of $\cR$ by the equivalence relation of staying bounded distance apart. The second assumption gives that every point of $\Gamma$ lies on an element of $\cR$. 

We use a visual metric $d$ on the Gromov boundary. Visual metrics have the property that there exists $0<a<1$ such that for all $E$ sufficiently large there exist  $C>0$ such that if $r$ is the maximal number such that $d(X(r), Y(r))\leq E$, then $$(1/C)a^{r} \leq d([X],[Y]) \leq C a^{r}.$$
Fix $E\geq 10D$,  $a$ and $C$ so this estimate holds. 

We will say that $X,Y\in R$ are $r$-adjacent if one of the following holds: 
\begin{enumerate}
\item $X(r)=Y(r)$,  or
\item $d(X(r), Y(r))=1$,  or
\item $d(X(r), Y(r+1))=1$, or 
\item $d(X(r+1), Y(r))=1$.
\end{enumerate} 
We will say they are at least $r$-adjacent if they are $s$-adjacent for some $s\geq r$. Similarly we will say they are more than $r$-adjacent if the same holds for some $s>r$. The key lemma is the following. 

\begin{lemma}\label{L:Adjacent}
There exists a constant $C_0$ such that if $X$ and $Y$ are $r$-adjacent, then there is a sequence $X=X_0, X_1, \ldots, X_n=Y$ such that 
$X_i$ and $X_{i+1}$ are at least $(r+1)$-adjacent for all $i$, and 
$$\diam(\{[X_0], [X_1], \ldots, [X_n]\})\leq C_0 a^r.$$
\end{lemma}

\begin{proof}
We proceed in cases as follows. In all cases the diameter bound follows immediately from the properties of the visual metric given above, so we merely indicate how to construct the desired sequence.

\bold{Case 1: $X(r)=Y(r)$.} Assumption \eqref{PLC:V} from the proposition's statement with $z=X(r)=Y(r)$ gives the existence of a path $x_0, \ldots, x_\ell$ from $x_0=X(r+1)$ to $x_n=Y(r+1)$.  For $0<i<\ell$, let $X_i$ be any point in $\cR$ that passes through $x_i$. Then 
$$X=X_0, X_1, \ldots, X_\ell=Y$$
is the desired path.

\bold{Case 2: $d(X(r), Y(r))=1$.} In this case we first apply assumption \eqref{PLC:adjacent} to get a path $x_0, \ldots, x_\ell$ from $x_0=X(r)$ to $x_\ell=Y(r)$,  and again let $X_i$ be any point in $\cR$ that passes through $x_i$. The geodesic ray $X_1$ can be chosen to pass through $X(r)$, and the geodesic ray $X_{\ell-1}$ can be chosen to pass through $Y(r)$. 

By Case 1, we can find a sequence $X=W_0, W_1, \ldots, W_k=X_1$ and a sequence $X_{\ell-1}=Z_0, \ldots, Z_m=Y$. The concatenation 
$$X=W_0, W_1, \ldots, W_k=X_1, X_2, \ldots, X_{\ell-1}=Z_0, \ldots, Z_m=Y$$
is the desired path  

\bold{Case 3: $d(X(r), Y(r+1))=1$.} Let $Z\in \cR$ be defined by $Z(i)=X(i)$ for $i\leq r$ and $Z(i)=Y(i)$ for $i>r$. So $X$ and $Z$ are $i$-adjacent and fall into Case 1, and $Z$ and $Y$ are $(i+1)$-adjacent. By Case 1, we can find a sequence  $X=W_0, W_1, \ldots, W_k=Z$, and the concatenation 
$$X=W_0, W_1, \ldots, W_k=Z, Y$$
is the desired path.

\bold{Case 4: $d(X(r+1), Y(r))=1$.} This is identical to the previous case.
\end{proof}

Let $X_0$ and $Y_0$ be points of $\cR$. Let $R$ be the maximal integer with $$d(X_0(R), Y_0(R))\leq E.$$ To prove the proposition, it suffices  to build a path from  $X_0$ to $Y_0$ of diameter at most $L a^{R}$, for some constant  $L$ not depending on $X_0$ or $Y_0$, and this is what we will do. 
%
%

We will iteratively define functions $\gamma_r: I_r\to  \cR$, where $$I_r\subset [0,1],\quad\quad r=0, 1, 2, \ldots$$ is an increasing nested family of finite sets and $\gamma_r|_{I_w}=\gamma_w$ for $w<r$. 

First pick a path $x_0, x_1, \ldots, x_n$ from $x_0=X_0(R)$ to $x_n=Y_0(R)$ of length $n\leq E$. Set $$I_0=\{i/n: i=0,\ldots, n\}.$$ Define $\gamma_0(0)=X_0$ and $\gamma_0(1)=Y_0$. For $0<i<n$, define $\gamma_0(i/n)$ to be any ray in $\cR$ passing through $x_i$. Note that $\gamma_0(i/n)$ and $\gamma_0((i+1)/n)$ are at least $(R-E)$-adjacent for all $i$, and that 
$$\diam(\gamma(I_0))\leq C_1 a^R$$
for a constant $C_1$ not depending on $X_0$ and $Y_0$. 

Now, assume we have constructed $I_r$ and $\gamma_r$ in such a way that adjacent points of $I_r$ map under $\gamma_r$ to points that are at least $(R-E+r)$-adjacent. We will define $I_{r+1}$ and $\gamma_{r+1}$ piece by piece on each interval of $[0,1]-I_r$. 

Suppose $(a,b)$ is an interval of $[0,1]-I_r$. So we have assumed that $\gamma_r(a)$ and $\gamma_r(b)$ are at least $(R-E+r)$-adjacent.  If these two points are more than $(R-E+r)$-adjacent, we set $$I_{r+1}\cap (a,b)=\{(a+b)/2\}$$ and $\gamma_{r+1}((a+b)/2)=\gamma_r(a)$. Otherwise  they are $(R-E+r)$-adjacent, and we use Lemma \ref{L:Adjacent} to produce a path $\gamma_r(a)=X_0, X_1, \ldots, X_n=\gamma_r(b)$. We set $$I_{r+1}\cap (a,b)=\{a+i(b-a)/n: i=1, \ldots, n-1\},$$ and define $\gamma_{r+1}(a+i(b-a)/n)=X_i$. (The $n$ here is not the same as the $n$ above.)

Set $I=\cup I_r$, and let $\gamma$ be the extension of the $\gamma_r$ to $I$. Tautologically we get 
$$\diam(\gamma(I))\leq \diam(\gamma(I_0)) + 2\sum_{r=1}^\infty \max_{x\in I_r} \min_{y\in I_{r-1}} d(\gamma(x), \gamma(y)).$$
Given the bounds previously obtained, this becomes 
$$\diam(\gamma(I))\leq C_1 a^R + 2\sum_{r=1}^\infty C_0a^{R-E+r}, $$
which is bounded by a constant times $a^R$. Since taking closure doesn't change the diameter of the set, we get the same diameter bound for the closure $\overline{\gamma(I)}$. 

A similar estimate shows that if $(a,b)$ is a component of $[0,1]-I_r$ then $\diam(\gamma(I)\cap [a,b])$ has diameter at most a constant times $a^r$. This implies that $\gamma$ is uniformly continuous as follows. Let $\epsilon>0$. Find $N$ such that if $(a,b)$ is a component of $[0,1]-I_N$ then $\diam(\gamma(I)\cap[a,b]) \leq \epsilon/2$. Let $\delta$ be smaller than the largest interval of $[0,1]-I_N$. If $x,y\in I$ have $|x-y|<\delta$, then we can find $z\in I_N$ within $\delta$ of both $x$ and $y$, and we conclude that 
$$d(\gamma(x), \gamma(y))\leq d(\gamma(x), \gamma(z))+ d(\gamma(z), \gamma(y)) \leq \epsilon/2+\epsilon/2=\epsilon.$$

Since $\gamma$ is uniformly continuous on $I$ it extends to a continuous function on $\overline{I}=[0,1]$. So $\overline{\gamma(I)}$ is connected, giving the result.
\end{proof}

\section{The essentially non-separating curve graph}\label{S:ENS}

The purpose of this section is to define a class of curves (and pairs of disjoint curves) called ``essentially non-separating", which for our purposes will be just as good as non-separating curves, and to justify our later claims that we can arrange for various curves to be essentially non-separating. Our motivation for looking at essentially non-separating curves can be seen by looking ahead to Section \ref{S:BGI}. 

\subsection{Definitions} Our main interest is punctured surfaces, but since cutting curves leads to surfaces with boundary, the following definition will be helpful. 

\begin{definition}\label{D:cut}
A cut surface is a connected surface of finite genus with a finite number of punctures and boundary components, together with a pairing of a subset of the boundary components\footnote{A pairing of a set is an equivalence relation where all equivalence classes have size two. The equivalence classes are called pairs. The boundary components in the subset with the equivalence relation are called paired, and the other boundary components are called unpaired.}, such that the surface admits a hyperbolic metric  where the punctures are cusps and the boundary consists of geodesics, and the surface is not a pants.

A cut surface has type $(h,m,p,u)$ if it has genus $h$, $m$ punctures,  $2p$ paired boundary components, and $u$ unpaired boundary components. 
\end{definition}

 Throughout this section, $\Upsilon=\Upsilon_{h,m,p,u}$ will denote a cut surface of type $(h,m,p,u)$. The condition that $\Upsilon$ is not pants excludes the cases when $h=0$ and $m+2p+u=3$. The reader should have in mind that $\Upsilon$ is a component of the complement of a multi-curve on a surface $\Sigma$ with no boundary components.

\begin{definition}
Let  $\hat{\Upsilon}$ denote the surface of genus $h+p$  with $m$ cusps and $u$ boundary components obtained by gluing each of the $p$ pairs. A multi-curve on $\Upsilon$ is called eventually non-separating if does not separate $\hat{\Upsilon}$. 
\end{definition}

\begin{definition}
A periphery of $\Upsilon$ is one of the punctures or boundary components. 
\end{definition}

\begin{definition}
A curve on $\Upsilon$ is called a pants curve if it bounds a genus 0 subsurface with no boundary components and two punctures.
\end{definition}

\begin{definition}\label{D:ENSgeneral}
A  curve $\alpha\in \cC \Upsilon$ will be called \emph{essentially non-separating}  if it is either 
 eventually non-separating or a pants curve.
\end{definition}

\begin{definition}\label{D:MultiENS}
A multi-curve $\alpha\cup \beta$ with two non-isotopic components will be called \emph{essentially non-separating} if $\alpha$ and $\beta$ are essentially non-separating, and either
\begin{enumerate}
\item $\alpha\cup \beta$ is eventually non-separating, or
\item at least one of $\alpha$ or $\beta$  is a pants curve, or
\item or $\alpha\cup \beta$ bounds a genus zero subsurface with no boundary components and 1 puncture. 
\end{enumerate}
\end{definition} 

\begin{definition}
The essentially non-separating curve graph $\cC_0 \Upsilon$ is the subgraph of $\cC \Upsilon$ whose vertices are essentially non-separating, with an edge from $\alpha$ to $\beta$ if $\alpha\cup \beta$ is essentially non-separating.  
\end{definition}

\subsection{Non-emptiness and 1-density.} Our first result on the essentially non-separating curve graph concerns its non-emptiness. 

\begin{lemma}\label{L:C0NotEmpty}
Suppose $h>0$ or $h=0$ and either $m\geq 2$ or $p\geq 1$. Then $\cC_0 \Upsilon \neq  \emptyset$.
\end{lemma}

For the proof, one should keep in mind that $\Upsilon$ has been assumed not to be a pants.

\begin{proof}
If $h>0$ then $\cC_0 \Upsilon$ has non-separating curves. 

If $h=0$  and $m\geq 2$, $\cC_0 \Upsilon$ has pants curves. 

If $h=0$ and  $p\geq 1$, $\cC_0 \Upsilon$   contains separating but eventually non-separating curves. 
\end{proof}

\begin{lemma}\label{L:subgraph}
Suppose that $\Upsilon$ is a subsurface of another cut surface $\Upsilon'$, and assume that each pair of boundary components of $\Upsilon$ are either a pair of boundary components of $\Upsilon'$ or bound an annulus in $\Upsilon'$. Then $\cC_0 \Upsilon$ is a subgraph of $\cC_0 \Upsilon'$.

If  $\alpha\in \cC_0\Upsilon$ and $\beta\in \cC_0 \Upsilon'$ and $\beta$ does not cut $\Upsilon$ and is not a boundary component of $\Upsilon$, then there is an edge in $\cC_0\Upsilon'$ from $\alpha$ to $\beta$. 
\end{lemma}

Here we assume that each puncture of $\Upsilon$ is a puncture of $\Upsilon'$, and as always we assume that the boundary components of $\Upsilon$ are essential in $\Upsilon'$. 

\begin{proof}
The first claim  follows direction from the definitions, after noting that an eventually non-separating multi-curve in $\Upsilon$ is also eventually non-separating in $\Upsilon'$.

The second claim is by definition if either $\alpha$ or $\beta$ is a pants curve. 
So it suffices to note that if both $\alpha$ and $\beta$ are essentially non-separating then so is $\alpha\cup \beta$. That follows from the fact that if $\Sigma$ is a connected surface and $U$ is a connected subsurface of $\Sigma$, and $\gamma$ is a non-separating curve on $U$, then $\gamma$ is non-separating on $\Sigma$, applied to $\Sigma=\hat{\Upsilon}'-\beta$ and $U=\hat{\Upsilon}$ and $\gamma=\alpha$. 
\end{proof}

\begin{lemma}\label{L:1dense}
Suppose either $h\geq 1$ or suppose $h=0$ and one of the following hold: 
\begin{enumerate}
\item $m\geq 3$, or
\item $p \geq 2$ and $m+2p+u\geq 5$, or
\item $p=1$ and $m=2$ and $u=0$. 
\end{enumerate}
Then 
$\cC_0\Upsilon$ is 1-dense in $\cC \Upsilon$.
\end{lemma}

By definition $1$-dense means that every vertex in $\cC \Upsilon$ is either in $\cC_0\Upsilon$ or adjacent to a vertex in $\cC_0\Upsilon$.

\begin{proof} 
Take $\alpha \in \cC \Upsilon$ not in $\cC_0\Upsilon$. So $\alpha$ is not eventually non-separating and is not a pants curve. 

$\Upsilon-\alpha$ has two components. If $h\geq 1$, one of these has positive genus, and so there is a non-separating curve disjoint from $\alpha$. So assume $h=0$. 

If there is a pair of punctures on one side of $\alpha$, then $\alpha$ is disjoint from the loop around that pair. So assume $m\leq 2$, and that there is at most 1 puncture on each side of $\alpha$.

If $p\geq 2$ and $m+2p+u\geq 5$, there is a paired curve on one component that also has another periphery that is not the other half of the pair. A loop around the paired curve and this periphery gives an eventually non-separating curve. 

The previous argument also works if $p=1, m=2, u=0$, since there is a puncture on each side. 
\end{proof}

\subsection{Connectivity.} We now turn to the connectivity of $\cC_0 \Upsilon$, where we break up the work into a number of lemmas. The first few concern the case $h\geq 1$.

\begin{lemma}\label{L:NonSepPath}
 There exists $C$ depending only on $\Upsilon$ such that the following holds. 
Suppose $\alpha, \beta \in \cC \Upsilon$ are non-separating. Then there is a sequence $$\alpha=x_0, x_1, \ldots, x_k=\beta$$ of non-separating curves, all of which are within distance $C$ of a geodesic from $\alpha$ to $\beta$ in $\cC \Upsilon$, such that: 
\begin{enumerate}
\item if $h\geq 2$, then $x_i$ is disjoint from $x_{i+1}$ for all $i$, and moreover $x_i\cup x_{i+1}$ is non-separating, and
\item if $h=1$, then the intersection number of $x_i$ and $x_{i+1}$ is equal to 1 for all $i$. 
\end{enumerate}
\end{lemma}

\begin{proof}[Proof sketch]
We start by noting that if $x, y$ are disjoint non-separating curves on a surface of genus at least 2, then either $x\cup y$ is non-separating, or there is a curve $z$ disjoint from both $x$ and $y$ such that both $x\cup z$ and $z\cup y$ are jointly non-separating. (To prove this, suppose that $x\cup y$ is separating. One component has positive genus, and we pick $z$ to be non-separating on that component.) This observation implies that in the $h\geq 2$ case it suffices to find a path where all the $x_i$ are non-separating, without worrying that the $x_i\cup x_{i+1}$ are non-separating, since one can then modify such a path by inserting curves between each pair $x_i, x_{i+1}$ as required to get the result. 

We also note that if $x$ and $y$ are disjoint non-separating curves on a torus, then there is a curve $z$ that is non-separating and intersects each of $x$ and $y$ once. Thus we get that in the torus case, it similarly suffices to find a sequence where the intersection number between $x_i$ and $x_{i+1}$ is at most 1. 

Let $\cP \Upsilon$ be the pants complex of $\Upsilon$. Let $P_\alpha$ be a pants decomposition containing the curve $\alpha$, and let $P_\beta$ be a pants decomposition containing the curve $\beta$. We will use the following fact: There is a path $$P_\alpha=P_0, P_1, \ldots, P_\ell=P_\beta$$ in the pants complex such that each curve of each $P_i$ is bounded distance from the geodesic from $\alpha$ to $\beta$. 

Let $n_i$ be a non-separating curve in $P_i$. Pick $n_0=\alpha$ and $n_\ell=\beta$. By definition of the pants complex, we have that that either $n_i$ and $n_{i+1}$ are disjoint, or intersect once and fill a one-holed torus, or  intersect twice and fill a four-holed sphere. 

First suppose $h\geq 2$. For each $i$ for which $n_i$ and $n_{i+1}$ fill a one-holed torus, let $r_i$ be a non-separating curve in the complement of that torus, and replace $n_i, n_{i+1}$ in the sequence with $n_i, r_i, n_{i+1}$.

For each $i$ for which $n_i$ and $n_{i+1}$ fill a four-holed sphere $S_{0,4}$, note that since the $n_i$ are non-separating, at least one of the 4 boundary components of $S_{0,4}$ must be non-separating. Let $r_i$ be a non-separating boundary component of $S_{0,4}$ and again replace $n_i, n_{i+1}$ in the sequence with $n_i, r_i, n_{i+1}$. This concludes the $h\geq 2$ case. 

In the $h=1$ case, it suffices to consider the pairs $n_i, n_{i+1}$ that fill a $S_{0,4}$. In that case, again we see that at least one boundary component $z$ must be non-separating, and we can replace $n_i, n_{i+1}$ with $n_i, z, n_{i+1}$ in the sequence. 
\end{proof}

\begin{lemma}\label{L:HighGenusConnected}
If $h\geq 2$, then $\cC_0 \Upsilon$ is connected. 

Moreover, there exists $C$ such that any two vertices of $\cC_0\Upsilon$ can be joined by a path in $\cC_0\Upsilon$ that stays within $C$ of a $\cC\Upsilon$ geodesic between the two points. 
\end{lemma}

\begin{proof}
We first note that every curve in $\cC_0 \Upsilon$ either is non-separating or is adjacent in $\cC_0 \Upsilon$ to a non-separating curve. This follows since a separating   curve divides the surface into two components, at least one of which must have positive genus. So, it suffices to show that any two non-separating curves  $x,y\in \cC_0 \Upsilon$ can be connected to each other. That follows from \ref{L:NonSepPath}. 
\end{proof} 

\begin{lemma}\label{L:Genus1Connected}
Suppose  $h=1$ and one of the following hold:
\begin{enumerate}
\item $m\geq 2$, 
\item $p\geq 1$ and $m+2p+u\geq 3$.
\end{enumerate} 
Then $\cC_0 \Upsilon$ is connected. 

Moreover, there exists $C$ such that any two vertices of $\cC_0\Upsilon$ can be joined by a path in $\cC_0\Upsilon$ that stays within $C$ of a $\cC\Upsilon$ geodesic between the two points. 
\end{lemma}
\begin{proof}
First note that every curve in $\cC_0 \Upsilon$  is equal to or adjacent in $\cC_0 \Upsilon$ to a non-separating curve. This follows since the complement of a separating curve must have a component that has genus 1, and that component has a non-separating curve. So, it suffices to show that any two non-separating curves  $x,y\in \cC_0 \Upsilon$ can be connected to each other. 

Let $x=x_0, x_1, \ldots, x_\ell=y$ be the path produced by Lemma \ref{L:NonSepPath}. Note that $\Upsilon-(x_i\cup x_{i+1})$ is connected for all $i$. 

For each $i$, let $y_i$ be a pants curve disjoint from $x_i\cup x_{i+1}$ (if $m\geq 2$) or a loop around one paired boundary component and another boundary component that isn't paired to the first (if $p\geq 1$ and $m+2p+u\geq 3$). Then $$x_0, y_0, x_1, y_1, x_2, y_1, \ldots, x_d$$ is a path in $\cC_0\Upsilon$.
\end{proof}

We now turn to techniques that will work when $h=0$. For any subset $G$ of the peripheries of $\Upsilon$, let $\cC_G\Upsilon$ be the set of curves in $\cC \Upsilon$ that bound a genus 0 surface containing two peripheries, both of which are in $G$. Thus $\cC_G\Upsilon$ corresponds naturally to the simple arcs joining two different elements of $G$. Note that $\cC_G\Upsilon$ is empty if $|G|<2$.

\begin{lemma}\label{L:ArcPath}
There exists $C$ depending only on $\Upsilon$ such that the following is true for any subset $G$ of the peripheries.  
Suppose $\alpha, \beta\in \cC_G \Upsilon$. Then there is a sequence $$\alpha=x_0, x_1, \ldots, x_\ell=\beta$$ of curves in $\cC_G \Upsilon,$ such that for all $i$ the arcs corresponding to $x_i$ and $x_{i+1}$ are disjoint except possibly at their endpoints, and such that each $x_i$ lies within $C$ of the geodesic from $\alpha$ to $\beta$. 
\end{lemma}

The proof of Lemma \ref{L:ArcPath} is contained in  \cite{HPW}. The condition that the arcs are disjoint except possibly at their endpoints can be rephrased by saying that one of the following is true for all $i$: 
\begin{enumerate}
\item $x_i$ and $x_{i+1}$ are disjoint,
\item exactly one periphery enclosed by $x_i$ is also enclose by $x_{i+1}$, and $x_{i}$ and $x_{i+1}$ have intersection number 2, or
\item $x_i$ and $x_{i+1}$ enclose the same peripheries and $x_i$ and $x_{i+1}$ have intersection number 4. 
\end{enumerate}

\begin{lemma}\label{L:Genus0Connected}
Suppose  $h=0$ and one of the following hold:
\begin{enumerate}
\item $p+m\geq 5$, or 
\item $p=4$ and $u\geq 1$ and $m=0$.
\end{enumerate}
Then  $\cC_0 \Upsilon$ is connected. 

Moreover, there exists $C$ such that any two vertices of $\cC_0\Upsilon$ can be joined by a path in $\cC_0\Upsilon$ that stays within $C$ of a $\cC\Upsilon$ geodesic between the two points. 
\end{lemma}

While we have not claimed that any of the sufficient conditions in this section are nescessary, we wish to emphasize that Lemma \ref{L:Genus0Connected} in particular is likely not sharp.

\begin{proof}
If $p+m\geq 5$, let $G$ be a set of peripheries containing one of each pair, together with all punctures. If $m=0, p=4, u\geq 1$, let $G$ be a set containing one of each pair and exactly one unpaired boundary. In all cases $|G|\geq 5$, $\cC_G\Upsilon \subset \cC_0 \Upsilon$, and there is an edge in $\cC_0 \Upsilon$ between any two disjoint curves in $\cC_G\Upsilon$. 

\begin{sublemma}
Every curve in $\cC_0\Upsilon$ is equal to or adjacent in $\cC_0\Upsilon$ to an element of $\cC_G\Upsilon$. 
\end{sublemma}
\begin{proof}
Consider a curve $\alpha \in \cC_0\Upsilon - \cC_G \Upsilon$. Since pants curves are in $\cC_G\Upsilon$, $\alpha$ must be eventually non-separating. Since $|G|\geq 5$, one side of $\alpha$ has at least three elements in $G$.

Suppose only a single pair $W$ is separated by $\alpha$. (This means $\alpha$ is ``only barely" eventually non-separating.) One side of $\alpha$ has at least three elements of $G$, so we can find a loop $\beta$ around two elements of $G$, neither of which is $W$. The curve $\alpha$ is adjacent to $\beta$ in $\cC_0\Upsilon$. 

Finally suppose that more than one pair is separated by $\alpha$. Again since one side of $\alpha$ has at least 3 elements of $G$, we can pick $\beta\in \cC_G\Upsilon$ adjacent to $\alpha$ in $\cC_0\Upsilon$. 
\end{proof}

\begin{sublemma}
Suppose $\alpha, \beta\in \cC_G \Upsilon$ are not disjoint but correspond to arcs which are disjoint except possibly at their endpoints. Then there is a curve $\gamma\in \cC_G \Upsilon$ disjoint from both $\alpha$ and $\beta$. 
\end{sublemma}

\begin{proof}
First suppose that $\alpha$ and $\beta$ have intersection number 4. In this case they fill an annulus containing two elements of $G$. Since $|G|-2\geq 3$, one side of this annulus contains at least 2 elements of $G$, and we can pick $\gamma$ to be a loop around two such elements of $G$. 

Next suppose that $\alpha$ and $\beta$ have intersection number 2. In this case they fill a disc containing three elements of $G$. The complement of this disc contains at least two elements of $G$, and we can pick $\gamma$ to be a loop around two such elements of $G$. 
\end{proof}

The result now follows easily from Lemma \ref{L:ArcPath} and the sublemmas. 
\end{proof}



  
  
  
 

\section{The Bounded Geodesic Image Theorem}\label{S:BGI}

We start by recalling the main tool for our analysis, the Bounded Geodesic Image Theorem of Masur and Minsky \cite{MMII}. 

\begin{theorem}\label{T:BGI}
For every connected surface $U$, there exists a constant $M$ such that if $V$ is a subsurface of $U$, and $a,b\in \cC U$, and 
$$d_V(a,b)>M,$$
then every geodesic from $a$ to $b$ contains a curve which does not cut $V$. 
\end{theorem}

Whenever we discuss a subsurface $V$ of a surface $U$, we assume that every component of the boundary $\partial V$ is essential and non-peripheral in $U$.  The quantity $d_V(a,b)$ denotes the distance in $\cC V$ between the subsurface projections to $V$ of $a$ and $b$. One says that a curve cuts $V$ if it cannot be isotoped outside of $V$. 

It is known that the constant $M$ can be taken to be independent of $U$ \cite{WebbUniformBGI}; we do not require that, but for notational simplicity we will assume we have fixed a constant $M$ that works for all $U$. 

In Section \ref{S:ENS}, we gave a number of definitions in the context of cut surfaces. For clarity, we now restate special cases of two of these definitions in the more typical context of surfaces with at most finitely many punctures.

\begin{definition}
A simple curve on a surface is called a pants curve if it bounds a genus 0 surface containing exactly two punctures. 
A curve is called essentially non-separating if it is non-separating or a pants curve. 
\end{definition}

One can rephrase the definition by saying that a pants curve is one whose complement contains at most one non-pants component. Following \cite{SchleimerEnd,DeadEnds,UniformKL, Recipe}, we record the following corollary of Theorem \ref{T:BGI}.

\begin{corollary}\label{C:BGI}
Assume $\Sigma$ is non-exceptional. Every $a\in S_r \cap \cC\Sigma$ that is essentially non-separating is adjacent to a vertex of $S_{r+1}$. Moreover, that vertex can be taken to be essentially non-separating. 
\end{corollary}

We include a proof since it inspired our analysis and we will rely frequently on the  techniques it uses. 

\begin{proof}[Proof of Corollary \ref{C:BGI}]
Note that the assumptions give that either $\Sigma-a$ is connected, or its complement consists of a three times punctured sphere and one other component. In the first case, set $V=\Sigma-a$, and in the second set $V$ to be the component of $\Sigma-a$ that is not a three times punctured sphere. Note that $V$ has infinite diameter curve complex. 

Consider any $b\in \cC V$ for which $d_V(b,c)>M$. Theorem \ref{T:BGI} gives that any geodesic from $b$ to $c$ must contain a curve that fails to cut $V$. The only curve that fails to cut $V$ is $a$ itself, so we see that any geodesic from $b$ to $c$ most pass through $a$. Since $a$ and $b$ are adjacent, this implies that $b\in S_{r+1}$.

Lemma \ref{L:1dense} allows us to choose $b$ to be essentially non-separating.
\end{proof}

\begin{remark}
Our use of the  Bounded Geodesic Image Theorem could be avoided, by replacing assumptions like $d_V(b,c)>M$ with assumptions like $d_V(b,c)>C d(b,c)$ and using that subsurface projections are coarse Lipschitz; compare to \cite[Lemma 2.2]{SchleimerEnd}. (Here $C$ is the Lipschitz constant.)
\end{remark}

\section{The key propositions}

In this section we state two key propositions, which will be proven in the next two sections, and apply them to obtain Theorem \ref{T:SphereConnected} and Theorem \ref{T:LinearlyConnected}. 

\subsection{The five times punctured sphere} The first of our key propositions is the following. 

\begin{proposition}\label{P:key05}
In $\cC \Sigma_{0,5}$, for any $c$ and any $r$ the following hold. 
\begin{enumerate}
\setlength\itemsep{0.15em}
\item\label{key05:V} For every  $z\in S_r$ and  $x, y\in S_{r+1}\cap B_1(z)$ there exists a path $$x=x_0, x_1, \ldots, x_\ell=y$$ with $$x_i\in (S_{r+1}\cup S_{r+2}\cup S_{r+3})\cap B_6(z)$$ for $0\leq i\leq \ell$.
\item\label{key05:adjacent} For every adjacent pair $x,y\in S_{r}$ there exists a path $$x=x_0, x_1, x_2, x_3, x_4=y$$ with $$x_i\in S_{r+1}\cup S_{r+2}$$ for $0<i<4$.
\end{enumerate}
\end{proposition}

Proposition \ref{P:key05} implies the low complexity cases of Theorem \ref{T:SphereConnected} as follows. 

\begin{proof}[Proof Theorem \ref{T:SphereConnected} in the cases $(g,n)\in \{(0,5), (1,2)\}$]
Since  $\cC\Sigma_{0,5}$ and $\cC\Sigma_{1,2}$ are isomorphic, it suffices to prove the result for $\cC\Sigma_{0,5}$. The result follows using Lemma \ref{L:BasicConnected}.
%
%
\end{proof}

\subsection{A subgraph avoiding dead ends.}\label{SS:Cc} In the higher complexity case, for technical reasons including the existence of dead ends \cite{DeadEnds}, we found it convenient to phrase many of our arguments in a rather specific subgraph of the curve graph. For $c\in \cC\Sigma_{g,n}$, define the subgraph $\cC_c \Sigma_{g,n}$ of $\cC\Sigma_{g,n}$ as follows. 
\begin{itemize}
\setlength\itemsep{0.15em}
\item A vertex $\alpha \in \cC\Sigma_{g,n}$ will belong to $\cC_c \Sigma_{g,n}$ if it is essentially non-separating  or if $\alpha=c$. 
\item If $\alpha, \beta$ are vertices of $\cC_c\Sigma_{g,n}$ that are joined by an edge in $\cC \Sigma_{g,n}$, we join them by an edge in $\cC_c \Sigma_{g,n}$ if either $d(c,\alpha)\neq d(c,\beta)$ or if $\Sigma_{g,n}-(\alpha\cup \beta)$ has at most one component that is not a three times punctured sphere.
\end{itemize}

Thus, the vertices of $\cC_c\Sigma_{g,n}$ are $\{c\}$ union the vertices of the essentially non-separating curve graph $\cC_0\Sigma_{g,n}$, and disjoint curves are joined by an edge if they have different distances to $c$ or if they are joined by an edge in $\cC_0\Sigma_{g,n}$.

\begin{lemma}\label{L:AlmostGeodesic}
Every $a\in S_r$ can be joined to $c$ by a geodesic $c=x_0, x_1, \ldots, x_r=a$ where all $x_i, i<r$ are in $\cC_c\Sigma$. 
\end{lemma}

Note that Lemma \ref{L:AlmostGeodesic} implies that $S_r^c= \cC_c\Sigma \cap S_r$ is the sphere of radius $r$ with center $c$ in the graph $\cC_c\Sigma$.

\begin{proof}
Let $c=x_0, x_1, \ldots, x_r=a$ be a geodesic from $c$ to $a$ with the minimal possible number of $x_i$ not in $\cC_c\Sigma$. Suppose in order to find a contradiction that some $x_j, 0<j<r$ is not in $\cC_c\Sigma$. 

Let $F$ be the subsurface filled by $x_{j-1}$ and $x_{j+1}$, so $F$ is connected. If some component of $\Sigma-F$ has at least two punctures, let $\gamma$ be a loop in $\Sigma-F$ around two punctures. Otherwise, let $\gamma$ be a non-separating curve in a component of $\Sigma-F$. 

 Replacing $x_j$ with $\gamma$ gives a contradiction to the fact that the geodesic was chosen to have the minimum number possible of vertices not in $\cC_c \Sigma$. 
\end{proof}

To account for the difference between $S_r$ and $S_r^c$, we have the following. 

\begin{lemma}\label{L:GetIntoCc}
Assume $\Sigma$ is not exceptional or low complexity. For all $a\in S_r$ not in $\cC_c\Sigma$ there exists $a'\in S_r\cap \cC_c\Sigma$ adjacent to $a$. 
\end{lemma}

\begin{proof}
By Lemma \ref{L:AlmostGeodesic} there exists $b\in S_{r-1}\cap \cC_c\Sigma$ adjacent to $a$. Suppose that $a$ is not in $\cC_c\Sigma$, so $\Sigma-a$ has two components, neither of which are a three times punctured sphere.  

Let $U$ be the component of $\Sigma-a$ not containing $b$. Let $a'\in \cC_c\Sigma$ be a curve in $\cC U$ for which $d_U(a', c)>M$, which can be found using Lemma \ref{L:1dense}. 

By Theorem \ref{T:BGI}, every geodesic from $a'$ to $c$ must contain some curve $e$ that fails to cut $U$. This $e$ must be disjoint from (or equal to) $a$ since it fails to cut $U$ and $a$ is on the boundary of $U$. If this $e$ has distance less than $r-1$ to $c$, that would imply that $a$ had distance at most $r-1$ to $c$, a contradiction. So $e$ has distance at least $r-1$ to $c$, and hence $a'$ has distance at least $r$ to $c$. Since $a'$ is disjoint from $b\in S_{r-1}$ the distance from $a'$ to $c$ is at most $r$, so we get $a' \in S_r$. 
\end{proof}

\subsection{The medium and high complexity cases.} For convenience we continue to use the notation $S_r^c=S_r\cap \cC_c\Sigma$. The second  key proposition is the following.

\begin{proposition}\label{P:keyhigher}
In the curve complex  $\cC \Sigma$ with $\Sigma$ medium or high complexity, for any $c$ and any $r$ we have the following. 
\begin{enumerate}
\setlength\itemsep{0.15em}
\item\label{keyhigher:V} For every  $z\in S_r^c$ and  $x, y\in S_{r+1}^c\cap B_1(z)$ there exists a path $$x=x_0, x_1, \ldots, x_\ell=y$$ in the graph $\cC_c \Sigma_{g,n}$ with $$x_i\in (S_{r+1}^c\cup S_{r+2}^c)\cap B_2(z)$$ for $0\leq i\leq \ell$.
If $\Sigma$ has high complexity, we can moreover obtain 
$$x_i\in S_{r+1}^c\cap B_1(z)$$ for $0\leq i\leq \ell$.
\item\label{keyhigher:adjacent} For every adjacent pair $x,y\in S_{r}^c$ there exists a path $$x=x_0, x_1, x_2=y$$ in $\cC_c \Sigma_{g,n}$ with $x_1\in S_{r+1}^c$.
\end{enumerate}
\end{proposition}

We can now explain how Proposition \ref{P:keyhigher} implies Theorem \ref{T:SphereConnected} in the medium and high complexity case. 

\begin{proof}[Proof Theorem \ref{T:SphereConnected} in the medium and high complexity cases]
Lemma \ref{L:BasicConnected} implies that the spheres $S_r^c$ in the graph $\cC_c \Sigma$ are connected in the high complexity case, and similarly for the union $S_r^c\cup S_{r+1}^c$ in the medium complexity case. Since Lemma \ref{L:GetIntoCc} gives that every element of $S_r$ is either in $\cC_c\Sigma$ or adjacent to an element of $S_r^c$, this gives the result. 
\end{proof}

\begin{remark}\label{R:OneComplexityAtATime}
Roughly speaking, we will prove Proposition \ref{P:keyhigher} by induction on the complexity of $\Sigma$. The proof for a given surface will make use of Theorem \ref{T:SphereConnected} for lower complexity surfaces. So it is important to note that really what we have just done is show that if Proposition \ref{P:keyhigher} holds for a surface $\Sigma$, then Theorem \ref{T:SphereConnected} holds for that same $\Sigma$. 
\end{remark}

\subsection{Linear connectivity.} We conclude by explaining why the propositions imply linear connectivity. 

\begin{proof}[Proof of Theorem \ref{T:LinearlyConnected}.]
Fix $x\in \cC\Sigma$. If $\Sigma=\Sigma_{g,n}$ has type $(g,n)\in\{(0,5),(1,2)\}$, we set $\Gamma=\cC\Sigma$. Otherwise, we set $\Gamma=\cC_c \Sigma$. Since $\Gamma$ is 1-dense in $\cC\Sigma$, the Gromov boundaries of $\Gamma$ and $\cC\Sigma$ are equal. 

Recall that \cite[Lemma 5.14]{MinskyI} gives that any point of of the Gromov boundary can be represented by a geodesic ray starting at $c$ in $\cC \Sigma$. This can be easily modified to lie in $\Gamma$, for example using the observation in the proof of Lemma \ref{L:AlmostGeodesic} first on the vertices of even index and then on the vertices of odd index. So assumption \eqref{P:LC:represented} in the statement of Proposition \ref{P:LinearlyConnectedGraph} is satisfied. Corollary \ref{C:BGI} gives that assumption \eqref{P:LC:1farther} is satisfied. 

Propositions \ref{P:key05} and \ref{P:keyhigher} show the remaining two assumptions are also satisfied, so Proposition \ref{P:LinearlyConnectedGraph} shows that $\Gamma$ has linearly connected boundary. 
\end{proof}

\section{The five times punctured sphere}\label{S:05}

In this section, let $\Sigma=\Sigma_{0,5}$.   
The goal of this section is to prove Proposition \ref{P:key05}. 

\subsection{Pentagons} We begin by recalling some basics on $\cC \Sigma$ when $\Sigma=\Sigma_{0,5}$. 

\begin{lemma}\label{L:no3or4cycles}
The curve complex $\cC \Sigma$ does not contain cycles of length 3 or 4.
\end{lemma}

Such cycles are called triangles and quadrilaterals respectively. 
%
%
%
We also note the following immediate consequence. 

\begin{corollary}\label{C:d2unique}
If $x$ and $y$ are vertices of $\cC \Sigma$ with $d(x,y)\leq 2$, then there is a unique geodesic from $x$ to $y$. 
\end{corollary}

Given the above, it is perhaps unsurprising that the geometry of $\cC \Sigma$ is often studied using pentagons. 

\begin{definition}
A pentagon in $\cC \Sigma$ is a 5 tuple of curves $(a_1, a_2, a_3, a_4, a_5)$ such that the 5 punctures of $\Sigma$ can be labeled by the elements of $\bZ/5\bZ$ in such a way that, for each $i$, the curve $a_i$
\begin{enumerate}
\setlength\itemsep{0.15em}
\item goes around  punctures $i$ and $i+1$, 
\item has intersection number 2 with $a_{i-1}$ and with $a_{i+1}$, and 
\item has intersection number 0 with $a_{i-2}$ and with $a_{i+2}$. 
\end{enumerate}
Traversing these curves in the cyclic order $(a_1, a_3, a_5, a_2, a_4)$ gives a 5-cycle in the graph $\cC \Sigma$. See Figure \ref{F:Pentagon}.
\end{definition}

\begin{figure}[h]
\includegraphics[width=0.6\linewidth]{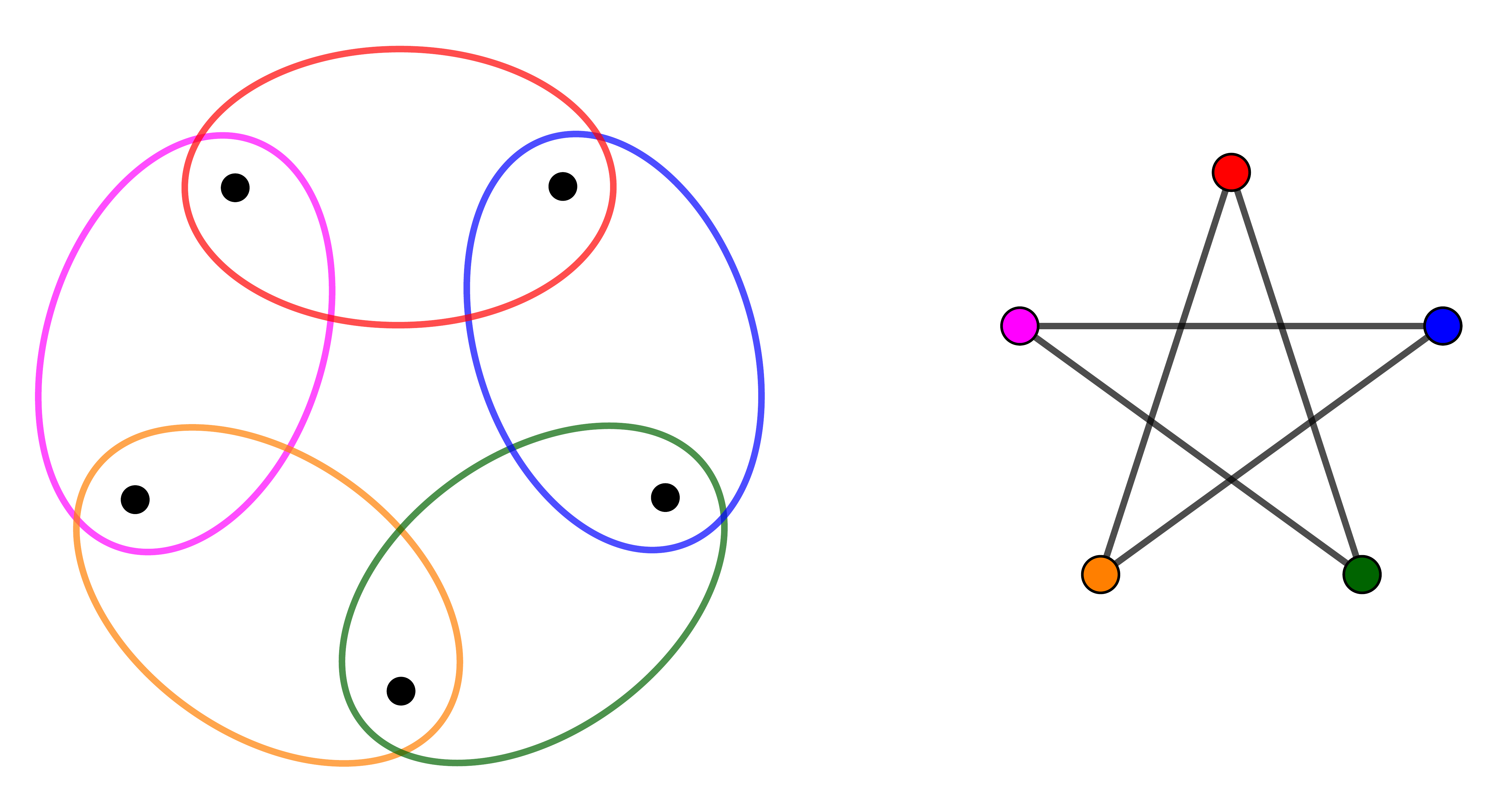}
\caption{The five curves in a pentagon (left), and the 5-cycle they define in $\cC \Sigma$ (right).}
\label{F:Pentagon}
\end{figure}  

It is known that every 5-cycle in $\cC\Sigma$ arises from a pentagon according to our definition \cite[Theorem 3.1]{FiniteRigid}, but  we do not require that result.

\begin{warning}
We will call the tuple $(a_1, a_2, a_3, a_4, a_5)$ a pentagon, and the tuple  $(a_1, a_3, a_5, a_2, a_4)$ a 5-cycle. Because of the different cyclic order, the terms pentagon and 5-cycle are not synonymous here. 
\end{warning}

The following two lemmas produce pentagons that will be helpful for our analysis.

\begin{lemma}\label{L:Pentagon2}
Suppose $a_1, a_3\in S_{r-1}$ are disjoint and distinct. Then there are curves $a_2, a_4, a_5\in S_r\cup S_{r+1}$ such that $(a_1, a_2, a_3, a_4, a_5)$ is a pentagon. 
\end{lemma}

\begin{proof}
Again, start with any $a_2, a_4, a_5\in \cC \Sigma$ such that $(a_1, a_2, a_3, a_4, a_5)$ is pentagon, and then replace it by its image under a large Dehn twist in $a_1$ and a large Dehn twist in $a_3$ to assume that 
$$d_{a_1}(a_2, c)>M, \quad\text{and}\quad
d_{a_3}(a_2, c)>M.$$ 
It suffices to show $a_2\in S_{r+1}$, so consider a geodesic from $a_2$ to $c$. Theorem \ref{T:BGI} gives that this geodesic must contain a curve that doesn't cut $a_1$, and that it must contain a curve that doesn't cut $a_3$. 

First consider the possibility that the geodesic contains a single curve $b$ that simultaneously  fails to cut both $a_1$ and $a_3$. The only such curves are $a_1$ and $a_3$ themselves, so without loss of generality assume $b=a_1$. We know $d(a_2, a_1)=2$, and the only thing adjacent to both $a_1$ and $a_2$ is $a_4$, so the geodesic must start $a_2, a_4, a_1$. This gives the result. 

Next assume the previous possibility does not occur. Say that, starting at $a_2$, the geodesic first contains a curve $b$ that doesn't cut $a_1$ and then contains a curve $b'$ that doesn't cut $a_3$, and neither of these curves are $a_1$ or $a_3$. Note $d(a_2, b')\geq 2$. Since $b'$ has distance 1 from $a_3\in S_{r-1}$, we're done if $d(a_2, b')\geq 3$. 

So assume $d(a_2, b')=2$. Thus $b$ comes immediately next to $a_2$ and hence is disjoint from both $a_2$ and $a_1$. Thus $b=a_4$. Since $b'$ is disjoint from $b=a_4$ and $a_3$, we get that $b'=a_1$, a contradiction. 
\end{proof}

%
%

\begin{lemma}\label{L:Pentagon3}
Suppose $a_3, a_4\in S_{r}$ have intersection number 2 and are both disjoint from $a_1\in S_{r-1}$. Then there are curves $a_2,  a_5\in S_r\cup S_{r+1}$ such that $(a_1, a_2, a_3, a_4, a_5)$ is a pentagon. 
\end{lemma}

\begin{proof}
Start with any $a_2, a_5$ that such that $(a_1, a_2, a_3, a_4, a_5)$ is a pentagon, 
%
%
then replace it with its image under a large Dehn twist about $a_1$ to assume that 
$$d_{a_1}(a_2, c)>M, \quad\text{and}\quad
d_{a_1}(a_5, c)>M.$$ 
We wish to show $a_2\in S_r\cup S_{r+1}$. Consider a geodesic from $a_2$ to $c$. Theorem \ref{T:BGI} gives that this geodesic must contain a curve $b$ that doesn't cut $a_1$. First assume that $d(a_2, b)=1$. In this case Corollary \ref{C:d2unique} gives that $b=a_4$, and we get $a_2\in S_{r+1}$. Next assume that $d(a_2,b)\geq 2$. The fact that $b$ doesn't cut $a_1\in S_{r-1}$ guarantees that $d(c,b)\geq r-2$. So we get $d(c,a_2)\geq r$ as desired. 
%
%
%
\end{proof}

\subsection{First notes on spheres} The very first observation on spheres is that $S_1$ has no edges, and so is as disconnected as possible. In analyzing larger spheres, the following definitions will be useful, where we continue to assume that a center vertex $c$ has been fixed. 

\begin{definition}
A vertex $x\in \cC \Sigma$ has unique backtracking if it has a unique neighbour $y$ with $d(y,c)=d(x,c)-1$. 
\end{definition}

\begin{definition}
A vertex $x\in \cC \Sigma$ has no sidestepping if it does not have any neighbour $y$ with $d(y,c)=d(x,c)$. 
\end{definition}

\begin{definition}
A vertex $x\in \cC \Sigma$ is forward facing if it has unique backtracking and no sidestepping. 
\end{definition}

The absence of edges in $S_1$ immediately gives the following. 

\begin{lemma}
Every $x\in S_1$ is forward facing. 
\end{lemma}

As a warm up, we can also easily observe the following. 

\begin{lemma}\label{L:HasForwardFacing}
Every sphere $S_r$ contains forward facing vertices. 
\end{lemma}

\begin{proof}
We assume $r>1$. For any $y\in S_{r-1}$, pick $x$ adjacent to $y$ with $d_{U}(x,c)$  much larger than $M$, where $U$ is the component of $\Sigma-y$ that isn't a three times punctured sphere. As noted in the proof of Corollary \ref{C:BGI}, we have that every geodesic from $x$ to $c$ goes through $y$, and in particular that $x\in S_r$. So $x$ has unique backtracking. 

To see that $x$ has no sidestepping, suppose in order to obtain a contradiction that $z\in S_r$ is adjacent to $y$. Since $z\neq y$, we see that $z$ cuts $U$. By the coarse Lipschitz property of subsurface projections, we see that $d_{U}(z,c)$ is approximately equal to $d_{U}(x,c)$, and in particular we get that $d_{U}(z,c)>M$. So any geodesic from $z$ to $c$ passes through $y$. Since $z\in S_r$ and $y\in S_{r-1}$, we see that actually $z$ and $y$ are adjacent. Thus $x,y,z$ form a triangle, giving a contradiction. 
\end{proof}

\begin{corollary}\label{C:disconnected}
For every $r$, the sphere $S_r$ is disconnected. 
\end{corollary} 

\begin{proof}
Lemma \ref{L:HasForwardFacing} gives that $S_r$ has vertices that are not adjacent to any other vertex of $S_r$. So the graph $S_r$ contains vertices with no edges.  
\end{proof}

We also make the following observations. 

\begin{lemma}\label{L:S12connected}
For any $c$, $S_1(c)\cup S_2(c)$ is connected. 
\end{lemma}

%
%

\begin{proof}
Consider two points in $S_1$. Connect them with a path $x_0, x_1, \ldots, x_k$ in the Farey graph $\cC (\Sigma-c)$ not passing through $c$. 

For each $0\leq i< n$, consider a 5 cycle $(x_{i}, y_i, z_i, x_{i+1}, c)$. The fact that $S_1$ does not have edges implies that $y_i, z_i\in S_2$. The path 
$$x_0, y_0, z_0, x_1, y_1, z_1, \ldots, x_{k-1}, y_{k-1}, z_{k-1} x_k$$
is contained in $S_1\cup S_2$ and proves the result. 
\end{proof}

\begin{lemma}\label{L:ForwardFacing}
Suppose $a\in S_r$ is forward facing. Then if $b, b'\in S_{r+1}$ are neighbours of $a$, then there is a path from $b$ to $b'$ in $(S_{r+1} \cup S_{r+2})\cap B_2(a)$.
\end{lemma}

\begin{proof}
The curve complex $\cC (\Sigma-a)$ is the Farey graph, and removing any given vertex does not disconnect the Farey graph. 
So we can find a path $b=x_0, x_1, \ldots, x_k=b$ in $\cC (\Sigma-a)$ not passing through the unique neighbour of $a$ in $S_{r-1}$. All the vertices of this path are in $S_{r+1}$ since $a$ has no sidestepping.


Using Lemma \ref{L:Pentagon3}, 
for each $i$ we can find a 5-cycle $(x_{i}, y_i, z_i, x_{i+1}, a)$ with $y_i, z_i\in S_{r+1}\cup S_{r+2}$.

 The path 
$$x_0, y_0, z_0, x_1, y_1, z_1, \ldots, x_{k-1}, y_{k-1}, z_{k-1} x_k$$
lies in $(S_{r+1} \cup S_{r+2})\cap B_2(a)$, proving the result. 
\end{proof}

\begin{lemma}\label{L:S23connected}
For any $c$, $S_2(c)\cup S_3(c)$ is connected. 
\end{lemma}
%
%

\begin{proof}
Since every element of $S_1$ is forward facing, Lemma \ref{L:ForwardFacing} gives that if $a\in S_1$, and $b, b'\in S_2$ are both neighbours of $a$, then there is a path from $b$ to $b'$ in $S_2\cup S_3$.

Now, given two points of $S_2\cup S_3$ that we wish to join with a path, we first join them by a path in $S_1 \cup S_2 \cup S_3$, which is possible by Lemma \ref{L:S12connected}. Since $S_1$ does not have edges, this path cannot contain consecutive vertices in $S_1$. The previous remark allows the path to be modified by replacing each entry of the path into $S_1$ by a path in $S_2\cup S_3$. 
\end{proof}

\subsection{The proof} Our key lemma is the following. 

\begin{lemma}\label{L:05key}
Suppose $a\in S_r$ and $b, b'\in S_{r+1}\cap B_1(a)$. Then there is a path $b, x_1, x_2, \ldots, x_k, b'$ in $\cC \Sigma$ with 
\begin{enumerate}
\setlength\itemsep{0.15em}
\item $d(x_i, a)\in \{2, 3\}$, 
\item $d(x_i, c)\geq r$,  
\item $d(x_i,c)=r$ implies $d(x_i,a)=2$, and the unique vertex adjacent to both $x_i$ and $a$ lies in $S_{r-1}$, and that $x_i$ has unique backtracking, and
\item if $a$ has unique backtracking then $d(x_i,c)=r$ also implies $x_i$ has no sidestepping. 
\end{enumerate}
\end{lemma}

\begin{proof}
Consider a path $x_1, \ldots, x_k$ in $S_2(a)\cup S_3(a)$, starting next to $b$ and ending next to $b'$, as must exist by Lemma \ref{L:S23connected}. Without loss of generality, by applying a large power of the Dehn twist in $a$, we can assume $d_a(x_i,c)$ is much greater than $M$ for all $i$. We will prove the lemma with this choice of path. 

The first point is true by construction. For the second point, suppose that $d(x_i, c)<r$. By the Bounded Geodesic Image Theorem, a geodesic from $x_i$ to $c$ must pass though $B_1(a)$. 
But any point in the geodesic cannot be in $B_1(a)$ because $B_1(a) \subset S_{r-1}\cup S_r\cup S_{r+1}$. This proves the second point. 

For the third point, we again use that any geodesic from $x_i$ to $c$ must contain a vertex $z$ in  $B_1(a)$; in fact $z$ must be the vertex adjacent to $x_i$ on the geodesic, and $z$ must be in $B_1(a)\cap S_{r-1}$. Thus $z$ lies on a geodesic from $x_i$ to $a$ of length 2. Since geodesics of length 2 are unique, there is only one possibility for $z$. 

For the last point, suppose that $y$ is a sidestep of $x_i$. There is a unique vertex $z$ adjacent to $x_i$ and $a$, and it is in $S_{r-1}$, so since $y\in S_r$ by assumption we get that $y\neq z$ and that $y$ cuts $a$. By the coarse Lipschitz property of subsurface projections, we get that $d_a(y,c)>M$. The same analysis as in the previous paragraph shows 
 that the geodesic from $y$ to $c$ must start with a curve in $S_{r-1}\cap B_1(a)$. Since $a$ has unique backtracking, the curve $z$ is the unique such curve, so we get that $y$ is adjacent to $z$. But now $y,x_i,z$ is a triangle, contradicting the fact triangles do not exist in $\cC \Sigma$. 
\end{proof}

\begin{lemma}\label{L:LinkAlmostConnected}
Suppose $a\in S_r$ and $b, b'\in S_{r+1}$ are neighbours of $a$. Then $b$ and $b'$ can be joined by a path in $(S_{r+1} \cup S_{r+2} \cup S_{r+3})\cap B_6(a)$. 
\end{lemma}

\begin{proof} 
We prove the lemma using a sequence of sublemmas. 

\begin{sublemma}\label{SL1}
Suppose $a\in S_r$ is forward facing and $b, b'\in S_{r+1}$ are neighbours of $a$. Then $b$ and $b'$ can be joined by a path in $(S_{r+1} \cup S_{r+2} \cup S_{r+3})\cap B_2(a)$. 
\end{sublemma} 
\begin{proof}
This is  Lemma \ref{L:ForwardFacing}. (Here we don't need $S_{r+3}$.)
\end{proof}

\begin{sublemma}\label{SL2}
Suppose $a\in S_r$ has unique backtracking and $b, b'\in S_{r+1}$ are neighbours of $a$. Then $b$ and $b'$ can be joined by a path in $(S_{r+1} \cup S_{r+2} \cup S_{r+3})\cap B_4(a)$. 
\end{sublemma} 
\begin{proof}
Lemma \ref{L:05key} gives the existence of a path in $(S_r \cup S_{r+1}\cup S_{r+2}\cup S_{r+3})\cap B_3(a)$ from $b$ to $b'$, with the extra property that each vertex on the path in $S_r$ is forward facing and is in $B_2(a)$.  Since forward facing vertices in particular have no side stepping, the path does not have adjacent vertices in $S_r$. 
%
%
 We can now replace the vertices in $S_r$ with paths given by Sublemma \ref{SL1} to obtain the result.
\end{proof}

To complete the proof, note that Lemma \ref{L:05key} gives the existence of a path in $(S_r \cup S_{r+1}\cup S_{r+2}\cup S_{r+3})\cap B_3(a)$ from $b$ to $b'$, with the extra property that each vertex on the path in $S_r$ has unique backtracking and is in $B_2(a)$. 

Using Lemma \ref{L:Pentagon2} we can modify this path to get the additional assumption that no two adjacent vertices are in $S_r$, resulting in a path in $(S_r \cup S_{r+1}\cup S_{r+2}\cup S_{r+3})\cap B_4(a)$. Sublemma \ref{SL2} can then be used to modify the path to avoid the vertices in $S_r$. 
\end{proof}

\begin{proof}[Proof of Proposition \ref{P:key05}]
The first claim follows from Lemma \ref{L:LinkAlmostConnected}, and the second claim from Lemma \ref{L:Pentagon2}.
\end{proof}

\section{The medium and high complexity cases}

The goal of this section is to prove Proposition \ref{P:keyhigher}. Throughout this section, we assume $\Sigma$ is medium or high complexity. We continue to use $\cC_0\Sigma$ to denote the essentially non-separating curve graph, and $\cC_c\Sigma$ to denote the result of adding the vertex $c$ to $\cC_0\Sigma$ as well as all edges between curves that are disjoint and have different distances to $c$ (see Section \ref{SS:Cc} for details).

\subsection{A note on the medium complexity case.} To help motivate why the medium complexity case has a slightly different statement, we give the following observation, which will not be used. 

\begin{lemma}
If $\Sigma=\Sigma_{0,6}$ and $c$ is a loop around a pair of punctures, then $S_1^c$ is disconnected.
\end{lemma}

\begin{proof}
Label the punctures by $\{1,2,3,4,5,6\}$. Every element of $\cC_c\Sigma$ is a loop around a pair of punctures, and we will call a loop around puncture $i$ and puncture $j$ an ``$ij$ loop". 

Suppose $c$ is a $12$ loop. Note that every element of $S_1^c$ is an $ij$ loop with $i,j \in \{3,4,5,6\}$. Note also that, in $S_1^c$, $34$ loops can only be adjacent to $56$ loops, $35$ loops can only be adjacent to $46$ loops, and $36$ loops can only be adjacent to $45$ loops. So $S_1^c$ has at least three connected components. 
\end{proof}


\subsection{The structure of the proof.}
For any $z\in S_r^c$, we define
$$\cO(z) = \{a \in S_1(z)\cap \cC_c\Sigma : d_U(a,c)>M\},$$
where $U$ is the unique component of $\Sigma-z$ that isn't a pants.

Note that Theorem \ref{T:BGI} implies that $\cO \subset S_{r+1}^c$. We think of $\cO$ as the subset of $S_1(z)$ which is most obviously in $S_{r+1}^c$. In subsequent subsections we will prove the following two lemmas. 

\begin{lemma}\label{L:IntoO}
For $z\in S_r^c$, any $x\in S_1(z)\cap S_{r+1}^c$ can be connected to $\cO(z)$ by a path in $S_1(z)\cap S_{r+1}^c$. 

 Moreover, for any $N$, any point in $\cO(z)$ can be connected by a path in $\cO(z)$ to a point $e$ with $d_U(e, c)>N$. 
\end{lemma}

\begin{lemma}\label{L:OConnected}
Suppose Theorem \ref{T:SphereConnected} holds for surfaces of smaller complexity than $\Sigma$.  
If $\Sigma$ is high complexity, then $\cO(z)$ is connected. If $\Sigma$ is medium complexity, then any two points of $\cO(z)$ can be joined by a path in $B_2(z)\cap (S_{r+1}^c \cup S_{r+2}^c)$. 
\end{lemma}

For both of these lemmas, it is implicit that we are working in (subgraphs of) the graph $\cC_c \Sigma$. Before we give the proofs, we explain how the two lemmas imply Proposition \ref{P:keyhigher}.

\begin{proof}[Proof of Proposition \ref{P:keyhigher}]
Suppose in order to find a contradiction that Proposition \ref{P:keyhigher} is false, and let $\Sigma$ be a minimal complexity surface for which it fails. Thus, Remark \ref{R:OneComplexityAtATime} gives that Theorem \ref{T:SphereConnected} can be assumed to hold for all surfaces of smaller complexity. 

Lemmas \ref{L:IntoO} and \ref{L:OConnected} imply \eqref{keyhigher:V} holds. 

To see that \eqref{keyhigher:adjacent} holds, suppose $x,y\in S_{r}^c$ are adjacent. By definition of the graph $\cC_c\Sigma$ and our complexity assumption on $\Sigma$, we know that there is a unique component $V$ of $\Sigma-(x\cup y)$ that is not a pants. 

%

By appealing to Lemma \ref{L:C0NotEmpty} and using that the mapping class group action on $\cC V$ has unbounded orbits,  we can find a curve $x_1\in \cC V$ that is essentially non-separating and such that $d_V(x_1, c)>M$. Then Theorem \ref{T:BGI} gives that every geodesic from $z$ to $c$ contains a curve not cutting $V$. The only curves not cutting $V$ are $x$ and $y$, so this shows that $x_1\in S_{r+1}^c$. This completes the verification of \eqref{keyhigher:adjacent}, showing that in fact Proposition \ref{P:keyhigher} actually does hold for $\Sigma$. 
\end{proof}

\subsection{Getting into $\cO(z)$.} We now turn to the proof of Lemma \ref{L:IntoO}, first giving another lemma on picking essentially non-separating curves. 

\begin{lemma}\label{L:abd}
Let $\Sigma$ have medium or high complexity. Let $a, b\in \cC \Sigma$ be disjoint and (individually) essentially non-separating. Then, there is a curve $d\in \cC \Sigma$ such that $d$ is disjoint from and not isotopic to $a$ and $b$, and such that $a\cup d$ and $b \cup d$ are essentially non-separating. If $V$ is a component of $\Sigma-(a\cup b)$ containing such a curve $d$, then the set of such $d$ is coarsely dense in $\cC V$. 
\end{lemma}

\begin{proof}
The last statement is immediate using the mapping class group action, so it suffices to find a single $d$. 
We proceed in cases, all of which are illustrated in Figure \ref{F:abd}.

\begin{figure}[h!]
\includegraphics[width=0.65\linewidth]{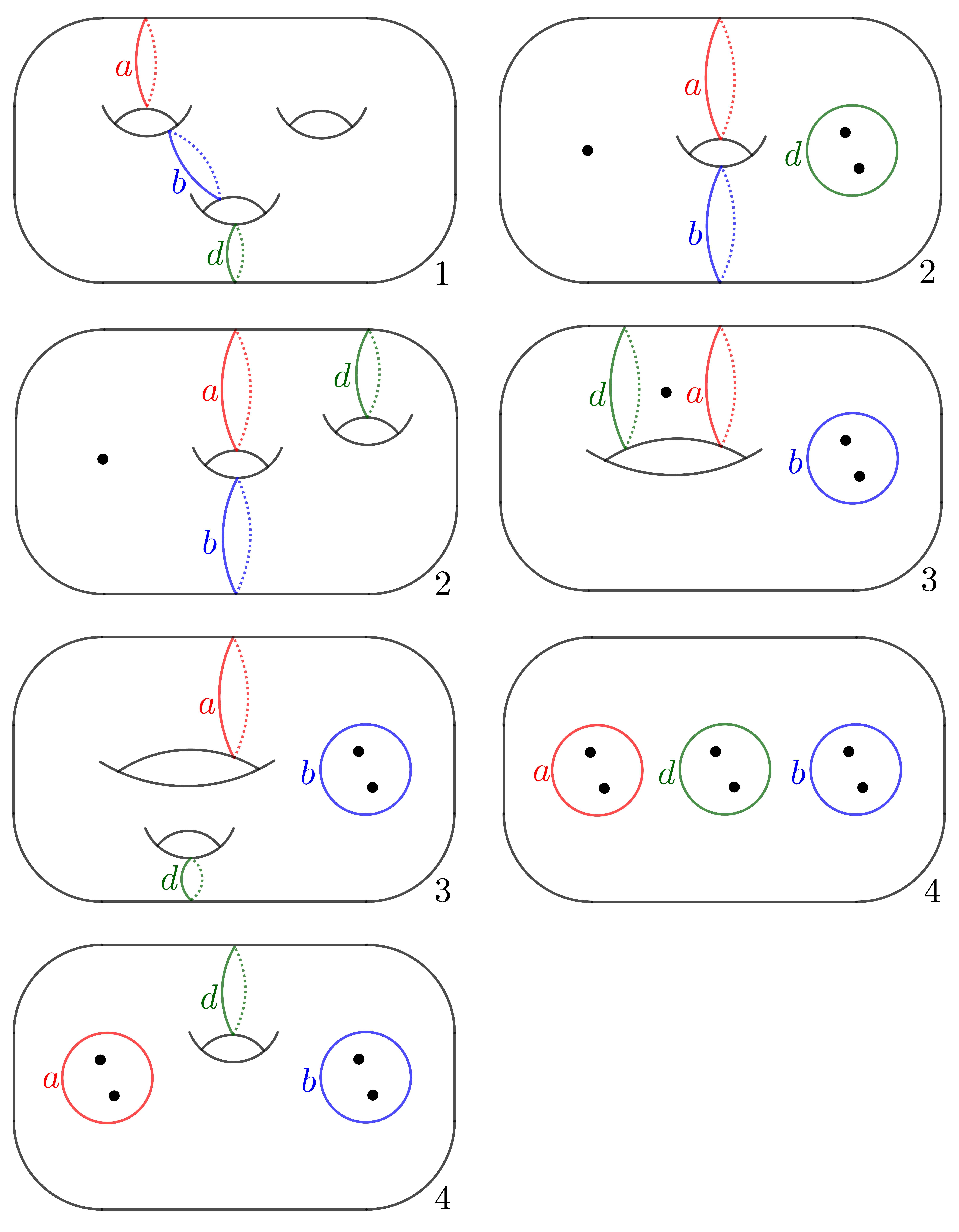}
\caption{The proof of Lemma \ref{L:abd}. The number below and to the right of each surface indicates which case it belongs to.}
\label{F:abd}
\end{figure}  

\bold{Case 1: $a, b$ and $a\cup b$ are all non-separating.} The existence of these curves implies  $g\geq 2$, and we can pick $d$ so that $a, b$ and $d$ bound a pants whose complement is connected. 

\bold{Case 2: $a$ and $b$ are both non-separating and $a\cup b$ is separating.} If $m\geq 3$, then we can pick $d$ to be a loop around two punctures on a component of $\Sigma-(a\cup b)$ with at least 2 punctures. If $m<3$ then $g\geq 2$ and we can pick $d$ to be a non-separating curve on a component of $\Sigma-(a\cup b)$ of positive genus. 

\bold{Case 3: $a$ is non-separating and $b$ is a pants curve (or vice versa).} If $m\geq 3$, we pick $d$ so that $a$ and $d$ bound a once punctured annulus. If $m \leq 2$ then $g\geq 2$ and we can pick $d$ so that $d$ and $a\cup d$ are non-separating. 

\bold{Case 4: $a$ and $b$ are both pants curves.} If $m\geq 6$, we can pick $d$ to also go around a pair of punctures. If $m\leq 5$ then $g\geq 1$ and we can pick $d$ to be non-separating. 
%
\end{proof}

\begin{proof}[Proof of Lemma \ref{L:IntoO}]
Let $x_0=x$, and let $U$ continue to denote the unique component of $\Sigma-z$ that isn't a pants.  We start with the following observation,  whose notation has been chosen to match how it will be applied. 

\begin{sublemma}\label{SL:sideways}
Suppose $x_i\in S_{r+1}^c \cap B_1(z)$ and $z\in S_r^c$. Let $V$ be a component of $\Sigma-(x_i\cup z)$, and suppose $x_{i+1}\in \cC V$ satisfies $d_V(c,x_{i+1})>M$. Then $x_{i+1}\in S_{r+1}$. 
\end{sublemma}

\begin{proof}
First note that since $x_i$ and $z_i$ are essentially non-separating and $\cC V$ is non-empty, the boundary of $V$ must contain both $x_i$ and $z$. In particular, every curve not cutting $V$ is disjoint from (or equal to) $x_i$ and $z$. Hence every curve not cutting $V$ must lie in $S_r$ or $S_{r+1}$. 

Theorem \ref{T:BGI} gives that the geodesic from $x_{i+1}$ to $c$ contains a curve not cutting $V$, giving the result. 
\end{proof}

Lemma \ref{L:abd} with $a=x_0$ and $b=z$ gives the existence of a curve $x_{1}$ disjoint from $x_0$ and $z$ such that $x_0\cup x_{1}$ and $x_1\cup z$ are essentially non-separating. The coarse density statement in Lemma \ref{L:abd} and Sublemma \ref{SL:sideways} imply that we can pick $x_1$ to be in $S_{r+1}^c$ as well.  

We now note the following statement.

\begin{sublemma}\label{SL:Uout}
For any $x_{i} \in S_{r+1}\cap \cC U$ such that $z\cup x_i$ is essentially non-separating there exists $x_{i+1}\in S_{r+1}\cap \cC U$ disjoint from $x_i$ such that $x_i\cup x_{i+1}$ and $z\cup x_{i+1}$ are essentially non-separating and such that $d_U(x_{i+1}, x_0)=d_U(x_i, x_0)+1$. 
\end{sublemma}

\begin{proof} 
Lemma \ref{L:abd} with $a=x_i$ and $b=z$ gives the existence of a curve $x_{i+1}$ disjoint from $x_i$ and $z$ such that $x_i\cup x_{i+1}$ and $z\cup x_{i+1}$ are essentially non-separating. Furthermore, if $V$ is the unique component of $\Sigma-(x_i\cup z)$ that isn't a pants, the coarse density statement in Lemma \ref{L:abd} implies we can assume that $d_V(x_{i+1}, x_0)>M$. 

Theorem \ref{T:BGI} implies that any geodesic from $x_{i+1}$ to $x_0$ in $\cC U$ must pass through a curve in $\cC U$ not cutting $V$. The only such curve is $x_{i}$, so we conclude $d_U(x_{i+1}, x_0)=d_U(x_i, x_0)+1$. 

Sublemma \ref{SL:sideways} implies that $x_{i+1}\in S_{r+1}^c$. 
\end{proof}

Repeatedly using the sublemma we obtain a path $x_0, x_1, x_2, \ldots$ in $\cC U \cap S_{r+1}^c$ which is a geodesic in $\cC U$ and hence eventually lies far from the projection of $c$ to $\cC U$. 

To get the second statement, note that if $x_0$ is already in $\cO(z)$ we can immediately use Sublemma \ref{SL:Uout} in the same way to get the desired result.  
\end{proof}

\subsection{Connectivity of $\cO(z)$.} Our final task in this section is to prove Lemma \ref{L:OConnected}. We start in high complexity. 

\begin{proof}[Proof of Lemma \ref{L:OConnected} when $\Sigma$ has high complexity.]
By Lemma \ref{L:IntoO}, it suffices to consider two points $x,y\in \cO(z)$ with $d_U(x,c), d_U(y,c)>M+C$, where $C$ is a large constant, and show they can be connected by a path in $\cO(z)$. 

Using that Theorem \ref{T:SphereConnected} is true for $U$, we can start by finding a path $$x=p_0, p_1, \ldots, p_\ell=y$$ in $\cC U$ from $x$ to $y$ with $d_U(p_i, c)>M+C$ for all $i$. 

Note that one of the following is true: 
\begin{enumerate}
\item $U$ has genus at least 2. 
\item $U$ has genus 1. If $\Sigma$ has genus 1, then the high complexity assumption $\xi(\Sigma)\geq 4$ implies that $\Sigma$ has at least 4 punctures, so $U$ has at least 2 punctures. If $\Sigma$ has genus 2, then $\Sigma$ has at least 1 puncture, so $U$ has a pair of boundary components and at least 1 puncture. 
\item $U$ has genus 0. If $\Sigma$ has genus 0, it has at least 7 punctures, so $U$ has at least 5 punctures. If $\Sigma$ has genus 1, it has at least 4 punctures, so $U$ has a pair of boundary components and at least 4 punctures. 
\end{enumerate}

Using Lemma \ref{L:1dense}, we can find an essentially non-separating $q_i\in \cC U$ which is equal to or adjacent to $p_i$. (We pick $q_0=x$ and $q_\ell=y$.) 

One of Lemmas \ref{L:HighGenusConnected}, \ref{L:Genus1Connected}, \ref{L:Genus0Connected} applies to give that there is a path in the essentially non-separating curve graph from $q_i$ to $q_{i+1}$ that has bounded diameter in $\cC U$. Assuming $C$ is large enough, every element of this path is in $\cO(z)$, so concatenating these paths gives the result. 
\end{proof}

The proof of the medium complexity case will be similar, but we first need the following. 

\begin{lemma}\label{L:MediumExtra}
Suppose $x_j, x_{j+1} \in \cO(z)$ have distance much more than $M$ to $c$. Suppose there exists a vertex $w$ in the essentially non-separating curve graph $\cC_0\Sigma$ not equal to $z$ that is adjacent in $\cC_0\Sigma$ to both $x_j$ and $x_{j+1}$. Then $w\in S_{r+1}^c\cup S_{r+2}^c$, and $w$ is adjacent in $\cC_c\Sigma$ to $x_j$ and $x_{j+1}$. 
\end{lemma}

Note that the assumption requires $w\cup x_j$ and $w \cup x_{j+1}$ to be essentially non-separating

\begin{proof}
The coarse Lipschitz property of subsurface projections implies that $d_U(c,w)>M$, so the geodesic from $w$ to $c$ must pass through $z$, proving the result. 
\end{proof}

Note that the medium complexity assumption that $(g,n)\in \{(2,0), (1,3), (0,6)\}$ gives rise to the following possibilities for $U$ in the notation of Definition \ref{D:cut}: 
$$(h,m,p,u)\in \{(1,0,1,0), (1,1,0,1), (0,3,1,0), (0,4,0,1)  \}.$$
Recall that $h$ denotes the genus of $U$, $m$ the number of punctures, $p$ the number of pairs of boundary components, and $u$ the number of unpaired boundaries. In particular, $U$ is always genus 0 or 1.

\begin{definition}
If $U$ is genus $0$, say $x_j, x_{j+1}\in \cC_0 U$ intersect nicely if they have intersection number 2.
If $U$ has genus 1, say $x_j, x_{j+1}\in \cC_0 U$ intersect nicely if they have intersection number 1. (This implies they are both non-separating.)  
\end{definition}

\begin{corollary}\label{C:MediumExtra}
Suppose that $\Sigma$ has medium complexity, and that $x_j, x_{j+1}\in \cC_{0} U$  that are essentially non-separating and intersect nicely. Then there exists $w\in S_{r+1}^c\cup S_{r+2}^c$ adjacent to both $x_j$ and $x_{j+1}$. 
\end{corollary}
\begin{proof}
The proof is illustrated in Figure \ref{F:w}.

\begin{figure}[h!]
\includegraphics[width=0.65\linewidth]{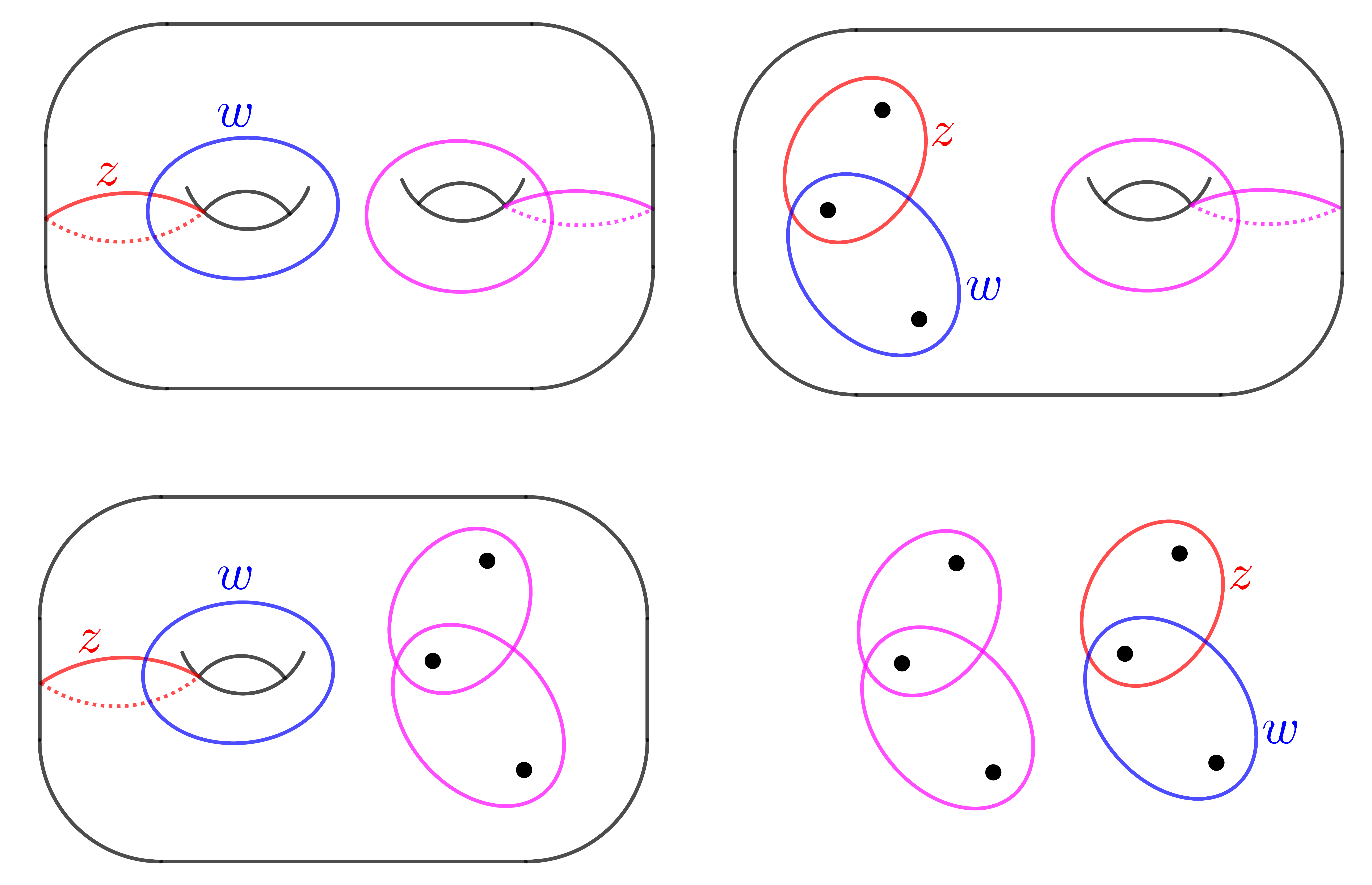}
\caption{The proof of Corollary \ref{C:MediumExtra}. The unlabelled curves are $x_j$ and $x_{j+1}$.}
\label{F:w}
\end{figure}  

First suppose $(g,n)=(2,0)$. In this case let $w$ be a non-separating curve disjoint from $x_j, x_{j+1}$ that intersects $z$ once. 

Next suppose $(g,n)=(1,3)$ and $z$ is a pants curve. In this case let $w$ be a pants curve disjoint from $x_j, x_{j+1}$ that intersects $z$ twice. 

Next suppose $(g,n)=(1,3)$ and $z$ is non-separating. In this case, let $w$ be a non-separating curve disjoint from $x_j, x_{j+1}$ that intersects $z$ once. 

Finally suppose $(g,n)=(0,6)$. In this case, let $w$ be a pants curve disjoint from $x_j, x_{j+1}$ that intersects $z$ twice.

With this $w$ as an input, Lemma \ref{L:MediumExtra} gives the result. 
\end{proof}

\begin{proof}[Proof of Lemma \ref{L:OConnected} when $\Sigma$ has medium complexity.]
The proof starts as in the high complexity case, giving a path $$x=p_0, p_1, \ldots, p_\ell=y$$ in $\cC U$ from $x$ to $y$ with $d_U(p_i, c)>M+C$ for all $i$, and we still have the existence of an essentially non-separating curve $q_i\in \cC U$ equal to or adjacent to $p_i$ for all $i$. (We pick $q_0=x$ and $q_\ell=y$.) It suffices to show that for each fixed $i$, there is a path from $q_i$ to $q_{i+1}$ in $B_2(z)\cap (S_{r+1}^c \cup S_{r+2}^c)$.

If $h=1$, then Lemma \ref{L:NonSepPath} gives the existence of a sequence $$q_i=x_0, x_1, \ldots, x_k=q_{i+1}$$ in $\cC_0 U$ with all $x_j$, $0<j<k $ non-separating, and $i(x_j, x_{j+1})=1$ for $0<j<k-2$. We can also assume that either $x_0$ is a pants curve and $x_1$ is disjoint from $x_0$, or that $x_0$ is a non-separating curve and $i(x_0, x_1)=1$, and similarly for $x_k$. Since Lemma \ref{L:NonSepPath} gives that the path has bounded diameter in $\cC U$, we can assume that $d_U(x_i, c)$ is much greater than $M$. Now, Corollary \ref{C:MediumExtra} gives that when $i(x_j, x_{j+1})=1$  there is a $w\in S_{r+1}^c\cup S_{r+2}^c$ adjacent to both $x_j$ and $x_{j+1}$. By interlacing $w$ of this form, we get a path as desired. 

If $h=0$, first note that if $a,b\in \cC_0 U$ have intersection number 4 and go around the same pair of peripheries, then there is some $v\in \cC_0 U$ such that $i(a,v)=2=i(b,v)$. Also note that if $a$ and $b$ are disjoint, then there is some $v\in \cC_0 U$ such that $i(a,v)=2=i(b,v)$.
This note and Lemma \ref{L:ArcPath}  gives the existence of a sequence $$q_i=x_0, x_1, \ldots, x_k=q_{i+1}$$ in $\cC_0 U$ such that for each $j$ the intersection number between $x_j$ and $x_{j+1}$ is 2. 
Again using Corollary \ref{C:MediumExtra} allows us to conclude.
\end{proof}

\bibliography{Spheres}{}
\bibliographystyle{amsalpha}

\end{document}